\documentclass[11pt]{amsart}
\usepackage{graphicx}
\usepackage[usenames]{color}
\usepackage{amssymb,amscd}
\newcommand\Z{\mathbb Z}

\newcommand\R{\mathbb R}
\newcommand\p{\mathcal K}
\newcommand\GG{\Gamma_d(q)}
\newcommand\GGG{\Gamma_3(q)}

\newcommand\Rinf{R_{\infty}}

\newcommand\vp{\varphi}
\newcommand\Op{\overline{\varphi}}
\newcommand\ti{t^{-1}}
\newcommand\ph{\varphi}
\newcommand\Zd{\Z^{d-1}}
\newcommand\Qi{Q_{i,j}}

\newcommand\LL{\bigoplus_{\mathbb L} \Z_q}
\newcommand\ZZZ{\bigoplus_{\Z^{d-1}} \Z_q}

\newtheorem{theorem}{Theorem}[section]
\newtheorem{proposition}[theorem]{Proposition}
\newtheorem{lemma}[theorem]{Lemma}

\title{Automorphisms of higher rank lamplighter groups}
\author{Melanie Stein, Jennifer Taback and Peter Wong}
\address{Department of Mathematics, Trinity College, Hartford, CT 06106}
\email{melanie.stein@trincoll.edu}
\address{Department of Mathematics,
Bowdoin College, Brunswick, ME 04011} \email{jtaback@bowdoin.edu}
\thanks{The second author acknowledges support from
National Science Foundation grant DMS-1105407 and Simons Foundation grant 31736 to Bowdoin College.  The authors would like to thank Manny Reyes and David Pollack for helpful conversations during the writing of this paper.}
\address{Department of Mathematics, Bates College, Lewiston, ME 04240} \email{pwong@bates.edu}

\begin{document}

\maketitle

\begin{abstract}
Let $\GG$ denote the group whose Cayley graph with respect to a particular generating set is the Diestel-Leader graph $DL_d(q)$, as described by Bartholdi, Neuhauser and Woess.  We compute both $Aut(\GG)$ and $Out(\GG)$ for $d \geq 2$, and apply our results to count twisted conjugacy classes in these groups when $d \geq 3$.  Specifically, we show that when $d \geq 3$, the groups $\GG$ have property $\Rinf$, that is, every automorphism has an infinite number of twisted conjugacy classes.  In contrast, when $d=2$  the lamplighter groups $\Gamma_2(q)=L_q = \Z_q \wr \Z$  have property $\Rinf$ if and only if $(q,6) \neq 1$.
\end{abstract}

\section{Introduction}
Bartholdi, Neuhauser and Woess in \cite{BNW} describe a family of groups $\GG$ whose Cayley graphs with respect to certain finite generating sets are horocyclic products of trees, or Diestel-Leader graphs $DL_d(q)$; see \cite{BNW} for their definition. The construction of the groups $\GG$ requires that for any prime $p$ dividing $q$, we have $d \leq p+1$. We call these groups Diestel-Leader groups.   When $d=2$ these groups are the well known lamplighter groups $L_q$; the identification of a Cayley graph of $L_q$ with the Diestel-Leader graph $DL_2(q)$ has been previously explored in \cite{BW}, \cite{TW} and \cite{Woess}.  Hence for $d \geq 3$, $\GG$ can be viewed as a higher rank generalization of the lamplighter groups, with the advantage that these higher rank groups are finitely presented.

Metric properties of Diestel-Leader groups were studied by the first two authors in \cite{ST}, where a method for computing word length with respect to the generating set $S_{d,q}$ is explicitly described.  Using this word length formula, these groups are shown to have dead end elements of arbitrary depth, infinitely many cone types and no regular language of geodesics. Random walks on Diestel-Leader graphs are explored in \cite{BNW}, where the authors also note that $\Gamma_d(q)$ is of type $F_{d-1}$ but not type $F_d$, and in most cases is an automata group.  For a more thorough introduction to the properties of $\GG$, we refer the reader to \cite{AR}, \cite{BNW} and \cite{ST}.

The construction of the Diestel-Leader groups given in \cite{BNW} and used in \cite{ST} is quite general;  in this paper we work with a specific family of Diestel-Leader groups.  This is made precise in Section \ref{sec:DL-groups}. Henceforth, $\GG$ will denote a member of this family.

Analogous to the lamplighter groups, the Diestel-Leader groups can be expressed as a split exact sequence
\begin{equation*}\label{seq:G}
1 \rightarrow \GG' \rightarrow \GG \rightarrow \Z^{d-1} \rightarrow 1.
\end{equation*}
In this paper we use the semidirect product structure to compute both the automorphism group  $Aut(\GG)$ and the outer automorphism group $Out(\GG)$. In \cite{BNW} graph automorphisms of a more general Diestel-Leader graph are studied and the full isometry group
of this graph is computed.  Although the graph $DL_d(q)$ has many natural symmetries which yield graph automorphisms, in general these symmetries do not induce group automorphisms. Namely, we prove the following theorem.

\noindent
{\bf Theorem 3.2}. If $d\geq 2$, then $$Aut(\GG) \cong Der(\Z^{d-1}, {\mathcal R}_d({\mathcal L}_q) ) \rtimes (U({\mathcal R}_d(\Z_q)) \rtimes {\mathcal K})$$ where
${\mathcal K}= \{ \beta  \in Aut(\Z^{d-1}) | K^{\beta}=K \}.$

\noindent
In the statement of the theorem, ${\mathcal R}_d({\mathcal L}_q)$ is a quotient of a polynomial ring in $d$ variables and their inverses with coefficients in $\Z_q$ and $U$ denotes the group of units; the group $K$ is the kernel of a particular short exact sequence given in Section \ref{subsec:presentation} used to define the derived group $\GG'$. In Theorem \ref{thm:S} we completely characterize the subgroup $\p$, and find that when $d > 3$, $\p$ is trivial unless $q=d-1$ is prime; moreover, if $d>3$ then any nontrivial automorphism corresponds to a permutation matrix that is a maximal length cycle in $\Sigma_{d-1}$. Although other possibilities may arise when $d=2$ or $d=3$, they are also quite constrained. Namely, if $d=2$ then $\p \cong \Z_2$, and if $d=3$, both $\p\cong\Z_2$ and $\p\cong D_6$ may arise.

 The outer automorphism group then has the following description.

\noindent
{\bf Theorem 3.7}. If $d \geq 3$, $$Out(\GG) \cong ( U({\mathcal R}_d(\Z_q))/ M) \rtimes \p,$$ where $M= \{ \Pi_{i=1}^{d-1} (t+l_i)^{x_i} | x_i \in \Z \}$ is the set of monomials with coefficient one.

In section 5, we extend these results to Baumslag's metabelian group.

As an application of these results we consider whether these groups have property $\Rinf$; a finitely generated group $G$ has this property if every automorphism $\vp \in Aut(g)$ has infinitely many twisted conjugacy classes. Two elements $g,h \in G$ are $\vp$-twisted conjugate if there is some $s \in G$ so that $sg\vp(s)^{-1} = h$.  In general, given an endomorphism $\ph :\pi \to \pi$ of a group $\pi$, the
$\ph$-twisted conjugacy classes are the orbits of the  action
of $\pi$ on itself via $\sigma \cdot \alpha \mapsto \sigma \alpha
\ph(\sigma)^{-1}$.

Groups with property $\Rinf$ include Baumslag-Solitar groups $BS(m,n)$ (excluding $BS(1,1)$) \cite{FG}, groups quasi-isometric to $BS(1,n)$ for $n > 1$ \cite{TW-BS} and generalized Baumslag-Solitar groups \cite{L}, non-elementary Gromov hyperbolic groups \cite{F,LL} as well as relatively hyperbolic groups \cite{FGD}, and mapping class groups \cite{FGD}, among others.

The following theorem follows from our characterization of the automorphism group of $\GG$.

\noindent
{\bf Theorem 4.3} The group $\GG$ has property $\Rinf$ for all $d \geq 3$.

\noindent
This is in contrast to the analogous result for the lamplighter groups $L_q = \Gamma_2(q)$, as it is proven in \cite{GW,TW} that $L_q$ has property $\Rinf$ iff $(q,6) \neq 1$.  This distinction is perhaps surprising because the proof in \cite{TW} relies on the geometry of the Diestel-Leader graph $DL_2(q)$.  However, when $d>2$, the limited number of automorphisms which can arise enable us to show that all such groups have property $\Rinf$.

Property $\Rinf$ has its roots in Nielsen fixed point theory.  Given a map
$f:X\to X$ of a compact connected manifold $X$, the fixed point set $Fixf=\{x\in X|f(x)=x\}$
is partitioned into fixed point classes, which correspond to the
$\ph$-twisted conjugacy classes of $\ph=f_{\sharp}$, the
homomorphism induced by $f$ on the fundamental group $\pi_1(X)$.

The nonvanishing of
the classical Lefschetz number $L(f)$ guarantees the existence of
fixed points of $f$, although it does not yield any information about the size of this set.  The Nielsen number $N(f)$ provides a lower bound on the size of this set, though is difficult to compute.   The Reidemeister number $R(f)$ is the cardinality  of the set of $\ph$-twisted conjugacy
classes, and is an upper bound for $N(f)$.  When $R(f)$ is finite,
this provides additional information about the cardinality of the
set of fixed points of maps in the homotopy class of $f$, and $R(f)$ is often easier to compute than $N(f)$.

For the family of {\em Jiang spaces} \cite{J}, the vanishing of $L(f)$ implies the vanishing of $N(f)$; the
nonvanishing of $L(f)$, combined with a finite Reidemeister number $R(f)$, yields the equality $R(f)=N(f)$.
This provides a valuable tool for calculating the Nielsen number in this particular case.  Groups $G$ which satisfy property $\Rinf$ will never be the
fundamental group of a manifold which satisfies the conditions
of a Jiang space.

\section{Diestel-Leader graphs and groups}\label{sec:DL-groups}
\label{sec:DLgroup}

In this section, we begin with a brief description of the Diestel-Leader graphs $DL_d(q)$ and the Diestel-Leader groups $\GG$ defined by  Bartholdi, Neuhauser and Woess in \cite{BNW}. We review the identification between the groups and the vertices of their Cayley graphs for background. We then focus on analyzing the groups as higher rank analogues of the lamplighter groups, extending the \lq\lq lampstand,\rq\rq and deriving group presentations which we use in later sections.

To define a Diestel-Leader graph $DL_d(q)$, let $T$ denote an infinite regular tree of valence $q+1$, oriented with one incoming edge and $q$ outgoing edges.  Let $T_1, T_2, \cdots ,T_d$ be $d$ copies of $T$. We define a height function $h_k: T_k\rightarrow \R$ for each $T_k$ for $k=1,2, \cdots ,d$ as follows.  Choose a basepoint vertex $v_k\in T_k$. Then there is a unique height function $h_k: {\it Vert}(T_k) \rightarrow \mathbb Z$ with $h_k(v_k)=0$ and satisfying the property that for any edge $e$ of $T_k$,  $h_k(i(e))-h_k(t(e))=1$, where $i(e)$ denotes the initial vertex of $e$ and $t(e)$ denotes the terminal vertex of $e$. Extending linearly across edges yields the desired height function $h_k: T_k\rightarrow \R$ .

The vertices of the Diestel-Leader graph $DL_d(q)$ are the following subset of $T_1 \times T_2 \times \cdots \times T_d$:
$$ Vert(DL_d(q)) = \{(t_1,\cdots ,t_d) \in Vert(T_1  \times \cdots \times T_d) | h_1(t_1) + \cdots + h_d(t_d) = 0 \}.$$

Two vertices $(t_1,t_2, \cdots ,t_d)$ and $(s_1,s_2, \cdots ,s_d)$ in $Vert(DL_d(q))$ are connected by an edge in $DL_d(q)$ if there are distinct indices $1 \leq i,j \leq d$ so that $t_m$ and $s_m$ are joined by an edge in $T_m$ for $m=i,j$ and $s_m=t_m$ for $m \neq i,j$.

Analogously, one can define a Diestel-Leader graph $DL_d(m_1,m_2, \cdots ,m_d) \subset T_1 \times T_2 \times \cdots \times T_d$ where the trees have valences $m_1+1,m_2+1, \cdots ,m_{d-1}+1$ and $m_d+1$ respectively.  However, if there exist $i,j$ with $m_i \neq m_j$ then the resulting Diestel-Leader graph is shown never to be the Cayley graph of any finitely generated group in \cite{EFW} when $d=2$ and in \cite{BNW} when $d \geq 3$.  Moreover, when $d=2$ and $m_1 \neq m_2$ this graph is not even quasi-isometric to the Cayley graph of a finitely generated group \cite{EFW}.

To construct the groups $\GG$, first let ${\mathcal L}_q$ denote a commutative ring of order $q$ with unit $1$ from which one can choose distinct elements $l_1,l_2, \cdots l_{d-1}$ whose pairwise differences $l_i-l_j$ for $i \neq j$ are all invertible.  Consider the ring of finite polynomials with coefficients in ${\mathcal L}_q$ over the variables $(t+l_1)^{-1}, \  (t+l_2)^{-1}, \cdots,(t+l_{d-1})^{-1},$ and $t$, which we denote by
$${\mathcal R}_d({\mathcal L}_q) = {\mathcal L}_q[(t+l_1)^{-1}, \ (t+l_2)^{-1}, \cdots,(t+l_{d-1})^{-1},t].$$
Polynomials in ${\mathcal R}_d({\mathcal L}_q)$ have the form
$$P=\sum_{{\bf v}=(v_1,v_2, \cdots ,v_{d-1}) \in \Z^{d-1} } a_{\bf v} \Pi_{i=1}^{d-1} (t+l_i)^{v_i}$$
where only finitely many coefficients $a_{\textbf{v}} \in {\mathcal L}_q$ are nonzero.

The group $\GG$ consists of matrices of the form
$$\left( \begin{array}{cc} (t+l_1)^{k_1} \cdots (t+l_{d-1})^{k_{d-1}} & P \\ 0 & 1 \end{array} \right), \text{ with }
k_1,k_2, \cdots ,k_{d-1} \in \Z \text{ and }P \in {\mathcal R}_d({\mathcal L}_q).$$
The generating set consisting of the following matrices is denoted $S_{d,q}$, and it is verified in \cite{BNW} that $\Gamma(\GG,S_{d,q}) = DL_d(q)$:
$$\left( \begin{array}{cc} t+l_i & b \\ 0 & 1 \end{array} \right)^{\pm 1}, \text{ with } b \in {\mathcal L}_q, \ i \in
\{1,2, \cdots ,d-1\} \text{ and }$$
$$ \left( \begin{array}{cc} (t+l_i)(t+l_j)^{-1} & -b(t+l_j)^{-1} \\ 0 & 1 \end{array} \right), \text{ with } b \in {\mathcal
L}_q, \ i,j \in \{1,2, \cdots ,d-1\}, \ i \neq j.$$

This group depends on the choice of $l_i \in {\mathcal L}_q$, and in fact for the same values of $d$ and $q$ but different choices of $l_i$, we may obtain non-isomorphic groups.  We will discuss different choices for these parameters later in this section.

When $d=2$, the above construction with ${\mathcal L}_q=\Z_q$ and $l_1=0$ yields a presentation for the lamplighter group $L_q = \Z_q \wr \Z$ with respect to the generating set $\left\{  \left(  \begin{array}{cc} t & b \\ 0 & 1 \end{array} \right)^{\pm 1}~~|~~ b \in \Z _q \right\} $.
  If we compare this to the standard presentation for the lamplighter group $L_q$, namely
$$L_q = \langle a,t | a^q, [a^{t^i},a^{t^j}] \rangle$$
we see that the matrix $ \left( \begin{array}{cc} t & b \\ 0 & 1 \end{array} \right)$  corresponds to the  $a^bt$.

When $d=3$, the group $\Gamma_3(q)$ (with ${\mathcal L}_q=\Z_q$, $l_2=1$, and $l_1=0$) is a torsion analog of Baumslag's metabelian group (BMG); the latter was introduced in \cite{Baum} as the first example of a finitely presented group with an abelian normal subgroup of infinite rank, namely its derived group.  This group has presentation
\begin{equation*} BMG= \langle a,s,t | st=ts, [a,a^t],aa^s=a^t \rangle. \end{equation*}
Baumslag also proved in \cite{Baum} that every finitely generated metabelian group can be embedded into a finitely presented metabelian group, and he embeds the lamplighter groups into a torsion analogue of the above group, namely the group with presentation
\begin{equation}\label{eqn:3generators} \langle a,s,t | a^q,st=ts, [a,a^t],aa^s=a^t \rangle. \end{equation}
Using the matrix representation of $\GGG$ given above with $l_1=0$ and $l_2=1$, we identify $\GGG$ with this torsion analogue of Baumslag's group, and the generators in the above presentation correspond to matrices as follows:
$$ a \leftrightarrow \left( \begin{array}{cc} 1 & 1 \\ 0 & 1 \end{array} \right), \ s \leftrightarrow \left( \begin{array}{cc} t & 0 \\ 0 & 1 \end{array} \right), \text{ and } t \leftrightarrow \left( \begin{array}{cc} 1+t & 0 \\ 0 & 1 \end{array} \right).$$
It is interesting to note that while the groups $\GG$ have a quadratic Dehn function for all values of $d$ and $q$ for which the construction holds \cite{dCT,W}, it is shown in \cite{KR} that Baumslag's metabelian group has an exponential Dehn function. Amchislava and Riley have recently given an exposition and overview of these groups and related contexts in which they appear in \cite{AR}.

\subsection{Identification between group elements and vertices in $DL_d(q)$.}
We briefly describe the identification between elements of $\GG$ and vertices of $DL_d(q)$, and refer the reader to \cite{AR} or \cite{BNW} for additional details.  This identification begins with a labeling of the vertices in each tree $T_i$ with an equivalence class of formal Laurent series in the variable $t+l_i$ (for $1 \leq i \leq d-1$), or the variable $t^{-1}$ (when $i=d$). Recall that a formal Laurent series in the variable $x$ with coefficients in the ring ${\mathcal R}$ is a series $\sum _{-\infty}^{\infty} r_i x^i $ with $r_i \in {\mathcal R}$ and $r_i=0$ for all but finitely many of the indices $i<0$. The set of all such series is denoted ${\mathcal R}((x))$.

When $1 \leq i \leq d-1$ define the following equivalence relation on the set of formal Laurent series ${\mathcal L}_q((t+l_i))$ :  given an integer $n$, two Laurent series ${\mathcal P},{\mathcal Q} \in {\mathcal L}_q((t+l_i))$ are equivalent under the relation if they agree on all terms of degree strictly less than $-n$.  Given a single series ${\mathcal P} \in {\mathcal L}_q((t+l_i))$  and $n \in \Z$, let $B_i({\mathcal P},q^{n})$ denote the equivalence class of such series containing ${\mathcal P}$, where the $i$ reflects the variable $t+l_i$.  Note that $B_i({\mathcal P},q^{n})$ is uniquely determined by the integer $n$ and the terms of ${\mathcal P}$ of degree at most $-n-1$, and thus every equivalence class has representatives which are (finite) polynomials.  It is straightforward to check that if ${\mathcal Q} \in B_i({\mathcal P},q^{n})$, then ${\mathcal P}$ and ${\mathcal Q}$ have identical terms of degree at most $-n-1$ and hence $B_i({\mathcal P},q^{n}) = B_i({\mathcal Q},q^{n})$.

When $i=d$ we alter this equivalence relation on ${\mathcal L}_q((t^{-1}))$ slightly.  Given ${\mathcal P} \in {\mathcal L}_q((t^{-1}))$ and $n \in \Z$, let $\bar{B}_d({\mathcal P},q^n)$ denote the set of all ${\mathcal P} \in {\mathcal L}_q((t^{-1}))$ so that ${\mathcal P}$ and ${\mathcal Q}$ agree on terms of degree at most $ -n$, that is, $\bar{B}_d({\mathcal P},q^n)$ is uniquely determined by the integer $n$ and the terms of ${\mathcal P}$ of degree at most $-n$.

For a fixed value of $i$, these equivalence classes have a natural partial order induced  by set containment which mimics the structure of the tree $T_i$.  An equivalence class $B_i({\mathcal P}, q^a)$ is the disjoint union of the set of $q$ equivalence classes $$\{ B_i({\mathcal P}+s(t+l_i)^{-a},q^{a-1})~~ | ~~s\in {\mathcal L}_q \},$$ so these form a natural set of labels for the $q$ vertices of $T_i$  adjacent to, but at height one above, the vertex labelled by $B_i({\mathcal P}, q^a)$.

Furthermore, note that if ${\mathcal Q} \in
B_i({\mathcal P}, q^a)$, then ${\mathcal P}$ and ${\mathcal Q}$ may certainly differ in the term of degree $-a$; let $r_{-a}(t+l_i)^{-a}$ be that term in ${\mathcal P}$ and $r'_{-a}(t+l_i)^{-a}$ be that term in ${\mathcal Q}$. Then  $$ B_i({\mathcal P}+s(t+l_i)^{-a},q^{a-1}) =  B_i({\mathcal Q}+s'(t+l_i)^{-a},q^{a-1}),$$ where $s'=s+r_{-a}-r'_{-a}$, so this disjoint union structure is independent of choice of equivalence class representative. Similarly, the equivalence class  $B_i({\mathcal P}, q^a)$ is always contained in the equivalence class  $B_i({\mathcal P}, q^{a+1})$, which is a natural label for the unique vertex adjacent to, but at height one below, the vertex labelled by $B_i({\mathcal P}, q^a)$. Thus, once an equivalence class is chosen to label one basepoint vertex of $T_i$, the partial order induces a unique labelling of the vertices of $T_i$.

We use this partial order to label each vertex in $T_i$ with an equivalence class of polynomials of the form $B_i({\mathcal P},q^{n})$ when $1 \leq i \leq d-1$ and $\overline{B}_d({\mathcal P},q^n)$ when $i=d$ as follows. We will choose a vertex $v=(v_1, \ldots, v_d)$ in $DL_d(q)$, with $h_i(v_i)=0$ for each $i$, to correspond to the group identity, and we then label the vertex $v_i \in T_i$ by $B_i(0,q^0)$ for $1 \leq i \leq d-1$, and label $v_d \in T_d$ by $\bar{B}_d(0,2^0)$. This induces a unique labeling of the remaining vertices of each tree $T_i$ by equivalence classes of Laurent series.
Thus, each vertex $(w_1, w_2, \ldots, w_d)$ of $DL_d(q)$  is labelled by a $d-tuple$
$$(B_1({\mathcal P}_1,q^{e_1}),B_2({\mathcal P}_2,q^{e_2}), \cdots ,\bar{B}_d({\mathcal P}_d,q^{e_d})),$$
where $h_i(w_i)=-e_i$  for $i=1,2, \cdots ,d$ and $e_d=-(e_1+e_2+ \cdots +e_{d-1})$.

To identify $g \in \GG$ with a vertex of $DL_d(q)$, let
\begin{equation}\label{eqn:g}g=\left( \begin{array}{cc} (t+l_1)^{k_1} \cdots (t+l_{d-1})^{k_{d-1}} & Q \\ 0 & 1 \end{array} \right) \in \GG.
\end{equation}

The following lemma, which is referred to in Section 3 of \cite{BNW},  allows us to identify the single variable polynomials related to $ Q$ which govern the identification between the group and this particular Cayley graph.
\begin{lemma}[Decomposition Lemma]
\label{lemma:decomposition}
Let $$Q \in {\mathcal R}_d({\mathcal L}_q)={\mathcal L}_q [ (t+l_1)^{-1},(t+l_2)^{-1}, \cdots , (t+l_{d-1})^{-1},t ]$$   where the $l_i \in {\mathcal L}_q$ are chosen so that $l_i-l_j$ is invertible whenever $i \neq j$. Then $Q$ can be written uniquely as $P_1(Q) + P_2(Q) + \cdots + P_d(Q)$ where
\begin{enumerate}
\item[(a)] for $1 \leq i \leq d-1$ we have that $P_i(Q)$ is a polynomial in $t+l_i$ all of whose terms have negative degree, and
\item[(b)] for $i=d$ we have that $P_d(Q)$ is a polynomial in $t^{-1}$ all of whose terms have non-positive degree.
\end{enumerate}
\end{lemma}

\begin{proof}
It is an easy exercise to see that any polynomial $Q$ can be written as the sum of $d$ polynomials of the desired form. To show such a decomposition is unique, suppose we have any decomposition $Q=Q_1+Q_2+ \cdots +Q_d$ where the $Q_i$ satisfy the conditions of the lemma. Any polynomial $Q \in {\mathcal R}_d({\mathcal L}_q)$ can be rewritten as a formal Laurent series ${\mathcal LS}_i( Q)$ in ${\mathcal L}_q ((t+l_i))$  for $1 \leq i \leq d-1$, or in ${\mathcal L}_q ((t^{-1}))$ when $i=d$.  We have $${\mathcal LS}_i( Q)=\sum_{j=1}^{d} {\mathcal LS}_i( Q_j);$$ note that ${\mathcal LS}_i( Q_i)=Q_i$. To compute ${\mathcal LS}_i( Q_j)$ for $i \neq j$ we use the formulas below.
When $k \in \Z^{-}$ we have:
$$(t+l_j)^k=\sum_{n=0}^{\infty} \left( \begin{array}{cc} k \\  n \end{array} \right) (l_j-l_i)^{k-n}(t+l_i)^n=\sum_{n=-k}^{\infty}\left( \begin{array}{cc}  k \\  k+ n \end{array} \right) l_j^{n+k}t^{-n},$$   and when $k \in {\mathbb Z}, k \geq 0$
$$(t)^k=\sum_{n=0}^{k} \left( \begin{array}{cc}  k \\  n \end{array} \right)(-l_i)^{k-n}(t+l_i)^n.$$
Notice that if $i \neq j$, then the minimal degree of ${\mathcal LS}_i( Q_j)$ is greater than or equal to zero for $i \neq d$, and strictly greater than zero for $i=d$.
Thus, $Q_i$ is the sum of the terms of  ${\mathcal LS}_i( Q)$ with negative degree in the case $1 \leq i \leq d-1$, and  $Q_d$ is the sum of the terms of ${\mathcal LS}_d( Q)$ with non-positive degrees, and hence the decomposition is unique.
\end{proof}

We identify $g=\left( \begin{array}{cc} (t+l_1)^{k_1} \cdots (t+l_{d-1})^{k_{d-1}} & Q \\ 0 & 1 \end{array} \right) \in \GG$ with the vertex of $DL_d(q)$ with label
$$(B_1( {\mathcal LS}_1(Q),q^{-k_1}), \ldots, B_{d-1}({ \mathcal LS}_{d-1}(Q),q^{-k_{d-1}}),\bar{ B}_d( {\mathcal LS}_d(Q),q^{-k_d})),$$
where $k_d=-(k_1+ \cdots+  k_{d-1})$. One checks that under this identification, each generator in the generating set $S_d(q)$ corresponds to a vertex adjacent to the vertex corresponding to the identity element in $DL_d(q)$, and it is straightforward to define a group action, namely if $h=\left( \begin{array}{cc} (t+l_1)^{j_1} \cdots (t+l_{d-1})^{j_{d-1}} & P\\ 0 & 1 \end{array} \right) \in \GG$, then
\begin{align*}
&h \cdot  (B_1(Q_1, q^{-k_1}), \ldots, \bar{ B}_d(Q_d,q^{-k_d}) \\&=
 (B_1({\mathcal LS}_1(P+M Q_1), q^{-k_1-j_1}), \ldots,  \bar{ B}_d({\mathcal LS}_d(P+M Q_{d}),q^{-k_d-j_d}))
\end{align*}
where $M=\Pi_{m=1}^{d-1} (t+l_m)^{j_m}$.

For $1 \leq i \leq d-1$, let $L_{i, k_i}(Q)$ be the sum of the terms of ${\mathcal LS}_i(Q)$ of degree less than $k_i$, and let  $L_{d, k_d}(Q)$ be the sum of the terms of ${\mathcal LS}_i(Q)$ of degree at most $k_d$. Then note that
the equivalence class $B_i({\mathcal LS}_i(Q) ,q^{-k_i})$ is uniquely determined by the polynomial $L_{i, k_i}(Q)$
 and the integer $k_i$, and the equivalence class  $\bar{B}_d(  {\mathcal LS}_d(Q)   ,q^{-k_d})$ is uniquely determined by the polynomial $L_{d, k_d}(Q)$ and the integer $k_d$.

We claim that the identification of elements of $\GG$ with vertices in $DL_d(q)$ is a one to one correspondence; that is, the element $g$ is determined by the vector $(k_1 \ldots, k_{d-1})$ and the polynomials  $L_{i,k_i}(Q)$.
To see this,   let $$ Q'=(t+l_1)^{-k_1} \cdots (t+l_{d-1})^{-k_{d-1}} Q.$$  Apply Lemma \ref{lemma:decomposition} to write $ Q'= P_1(Q') + P_2(Q') + \cdots + P_d(Q')$. Then
$$Q=(\Pi_{m=1}^{d-1} (t+l_m)^{k_m})  P_1(Q') +( \Pi_{m=1}^{d-1} (t+l_m)^{k_m}) P_2(Q') + \cdots + (\Pi_{m=1}^{d-1} (t+l_m)^{k_m}) P_d(Q'),$$ so the vector $(k_1, \ldots, k_{d-1})$ and the polynomials  $ P_i(Q')$ for $1 \leq i \leq d$ determine the original element $g$.  Note that for $1 \leq i \leq d-1$ and $j \neq i$, $${\mathcal LS}_i((\Pi_{m=1}^{d-1} (t+l_m)^{k_m})  P_j(Q') )$$ consists of terms in $t+l_i$ of degree greater than or equal to $k_i$. Thus $${\mathcal LS}_i((\Pi_{m=1}^{d-1} (t+l_m)^{k_m})  P_i(Q'))$$ is the sum of
 $L_{i, k_i}(Q)$ and terms of degree greater than or equal to $k_i$. But then since  $${\mathcal LS}_i((\Pi_{m=1}^{d-1} (t+l_m)^{-k_m}) (t+l_i)^{x})$$  has only terms of nonnegative degree if $x \geq k_i$ , $ P_i(Q')$ is simply the sum of the terms of negative degree which occur in ${\mathcal LS}_i((\Pi_{m=1}^{d-1} (t+l_m)^{-k_m})L_{i,k_i} )$, so $P_i(Q')$ can be recovered from $L_{i, k_i}(Q)$ and the vector $(k_1 \ldots, k_{d-1})$ .
 Similarly,  for $j \neq d$, $${\mathcal LS}_d((\Pi_{m=1}^{d-1} (t+l_m)^{k_m})  P_j(Q') )$$ consists of terms in $t^{-1}$ of degree greater than $k_d$, so
this Laurent series is the sum of  $L_{d, k_d}(Q)$ and terms of degree greater than $k_d$. Thus, $P_d(Q')$ can be recovered from $L_{d, k_d}(Q)$ and the vector $(k_1 \ldots, k_{d-1})$. Therefore, the element $g$ is determined by the polynomials  $L_{i,k_i}(Q)$ and the vector $(k_1 \ldots, k_{d-1})$, as claimed.

\subsection{Choosing parameters}{\label{parameters}
We now describe the specific Diestel-Leader groups which we consider for the remainder of the paper. For fixed $d$ and $q$ (always satisfying $d \leq p+1$ for any prime $p$ dividing $q$), the choice of the coefficient ring ${\mathcal L}_q$ and the elements $l_i \in {\mathcal L}_q$ satisfying $l_i-l_j$ is invertible when $i \neq j$ determine (possibly) non-isomorphic groups whose Cayley graphs have the same underlying graph.
Note that for any such $d$ and $q$,
there is a choice of $l_i$ satisfying the required condition with ${\mathcal L}_q = \Z_q$, that is, producing a group which satisfies the construction in \cite{BNW}.  For instance, if we choose $l_i=i-1$, then the fact that $d-1 \leq p$ for every prime $p$ dividing $q$ implies that $l_i-l_j$ is invertible in $\Z_q$ when $i \neq j$.

For the remainder of this paper $\GG$ will denote a Diestel-Leader group with  ${\mathcal L}_q = \Z_q$, and with $l_1=0$ and some choice of $l_2, \cdots ,l_{d-1}$ satisfying $l_i-l_j$ is invertible when $i \neq j$.

\subsection{ Extending the ``lamplighter picture" to $\GG$ for $d \geq 3$}

It is easy to see for any $d \geq 2$  that $\GG'$ consists of those matrices of the form $\left( \begin{array}{cc} 1 & P \\ 0 & 1 \end{array} \right)$, and that this group is infinitely generated by conjugates of $a=\left( \begin{array}{cc} 1 & 1 \\ 0 & 1 \end{array} \right)$ by products of the form
\begin{equation*}\label{eqn:conjugate}\Pi_{i=1}^{d-1} \left( \begin{array}{cc} t+l_i & 0 \\ 0 & 1 \end{array} \right)^{x_i}\end{equation*}
where $x_i \in \Z$.  (The notation $a$ comes from the presentation \eqref{eqn:3generators} for $\GGG$.)

Recall that an element of $\Gamma_2(q)=L_q$ can be viewed as a pair $(A,t)$, where $A \in \bigoplus_{\Z} \Z_q$ and $t \in \Z$.  In fact, each pair $(A,t)$ corresponds to a unique element of $$\Gamma_2(q)=L_q = \left( \bigoplus_{\Z} \Z_q \right) \rtimes \Z = L_q' \rtimes \Z,$$   and the pair  $(A, t)$ represents an element of $L_q'$ if and only if $t=0$. When we view $g \in L_q$ as a pair $(A,t)$, we can interpret $\bigoplus_{\Z} \Z_q$ as an infinite string of  q-way ``lightbulbs," or copies of $\Z_q$, placed along a ``lampstand," or copy of $\Z$; an element of this sum is viewed as a finite collection of illuminated bulbs, where each bulb has $q-1$ possible illuminated states.  The integer $t$ corresponds to the position of the ``lamplighter."

Similarly, each element of $\GG = \GG' \rtimes \Z^{d-1}$ for $d\geq 3$ can be represented by a pair $(A,{\bf x})$, where $A \in \bigoplus_{\Zd} \Z_q$ and ${\bf x} \in \Zd$ as follows.
Suppose  that $g \in \GG$ has matrix $\left( \begin{array}{cc}  \Pi_{i=1}^{d-1}(t+l_i)^{x_i} & Q \\ 0 & 1 \end{array} \right)$. Note that  $Q$ has the form
$$Q=\sum_{j=1}^r c_j \Pi_{i=1}^{d-1} (t+l_i)^{v_{j,i}}$$
with $c_j \in \Z_q$, but this expression is not unique.
Let ${\bf v}_j = (v_{j,1},v_{j,2},v_{j,3}, \cdots ,v_{j,d-1}) \in \Z^{d-1}$.  Then
$$A=(c_1)_{{\mathbf v}_1} \oplus (c_2)_{{\mathbf v}_2} \oplus  \cdots (c_r)_{{\mathbf v}_r} \in \bigoplus_{\Zd} \Z_q, $$
where $c_{{\mathbf v}}$ denotes the element $c$ in the copy of $\Z_q$ indexed by ${\mathbf v} \in \Z^{d-1}$, and  ${\mathbf x} = (x_1,x_2, \cdots ,x_{d-1})$.  Since the expression above for $Q \in {\mathcal R}_d({\mathbb Z}_q)$ is not unique, $A \in \bigoplus_{\Zd} \Z_q$ is not unique, although the vector ${\bf x}$ is uniquely determined by $g$.

For example, assume that $d=3,\ q \geq 2,  \ l_1=0, \ l_2=1$, and $$Q=(1+t)^1 + t^2 = 1 + t + t^2.$$  Then $g = \left( \begin{array}{cc} 1 & Q \\ 0 & 1 \end{array} \right)$ is represented by both $(A_1, {\bf 0})$ and $(A_2, {\bf 0})$ where $A_1=1_{(0,1)} \oplus 1_{(2,0)} \in \bigoplus_{\Z^2} \Z_q$ and $A_2=1_{(0,0)} \oplus 1_{(1,0)} \oplus 1_{(2,0)} \in \bigoplus_{\Z^2} \Z_q$.

Reversing this procedure, it is easy to obtain a matrix representing a unique group element of $\GG$ from any pair $(A, {\mathbf x}) \in \bigoplus_{\Z^{d-1}} \Z_q \rtimes \Z^{d-1}$.
 We can again interpret $\bigoplus_{\Z^{d-1}}\Z_q$ as an infinite array of  ``q-way light bulbs," or copies of $\Z_q$ placed on the $d-1$ dimensional grid. The coefficients in $Z_q$ at each point on the grid specify the state of the light bulb, with a zero coefficient indicating that the given bulb is not illuminated. As with the lamplighter group, the vector ${\mathbf x}$ can still be thought of as specifying the position of the lamplighter.

Recall that $Q$ does have a unique decomposition $Q=P_1(Q) + \cdots +P_{d-1}(Q)+P_d(Q)$ given in Lemma \ref{lemma:decomposition}, where each $P_i(Q)$ is a polynomial in a single variable. Using this decomposition to choose $\alpha \in  \bigoplus_{\Zd} \Z_q$, we obtain a unique representative  $(A, {\mathbf x}) \in {\LL} \rtimes \Z^{d-1} \subset \bigoplus_{\Z^{d-1}} \Z_q \rtimes \Z^{d-1}$, where the ``lampstand" ${\mathbb  L}$ is the union of $d$ rays:
$${\mathbb  L}={\mathbb  L}_1 \cup {\mathbb  L}_2 \cup\cdots \cup {\mathbb  L}_d \subset \Z^{d-1},$$ with $${\mathbb  L}_i=\{ (0, \ldots, 0, a_i, 0, \ldots, 0)) | a_i \in \Z^- \}~ \text{ for} ~ 1 \leq i \leq d-1,$$and $${\mathbb  L}_d=\{ (a_d, 0, \ldots, 0) | a_d \in  \Z^+ \cup \{0\} \}.$$ This induces a natural surjection from $\ZZZ$ to $\LL$, taking any $A$ corresponding to a decomposition of a polynomial $Q$ to the unique element of $\LL$ determined by the unique decomposition of $Q$ given in Lemma \ref{lemma:decomposition}.  It follows that we have a bijection between elements of $\GG$ and elements of $\left( \LL \right) \rtimes \Z^{d-1}$.

This generalizes both the standard lampstand construction for the lamplighter groups, and the construction for $\Gamma_3(2)$ given in \cite{CR} and \cite{CR1}.

\subsection{A presentation for $\GG$}\label{subsec:presentation}

We first obtain a presentation for $\GG'$. Let   $\{ t_1, \ldots, t_{d-1}\}$ be generators  for $\Z^{d-1}$, and then the element $\Pi_{i=1}^{d-1} t_i^{v_i}$  can be represented by the vector ${\bf v}=(v_1, \ldots, v_{d-1})$. Recall from the previous section that if $b \in \Z_q$ and ${\bf v}$ denotes an element of $\Z^{d-1}$, then $b_{\bf v} \in \ZZZ$ denotes the element $b$ in the copy of $\Z_q$ indexed by ${\bf v}$. Then the group $\ZZZ$ has the following presentation:
$$\ZZZ =\langle  1_{\bf x}, {\bf x} \in \Z^{d-1} |q( 1_{\bf x}),  [1_{\bf x}, 1_{\bf y}] \rangle.$$

Now recall the natural surjection from $\ZZZ$ to $\LL$, and let $K$ be the kernel of this map, so we have \begin{equation*}\label{seq:commutator}
1 \rightarrow K \rightarrow \ZZZ \rightarrow \LL \rightarrow 1.
\end{equation*}

Recall that the formal variables used to define the ring ${\mathcal L}_q({\mathbb Z}_q)$ are $t, t+l_2, \cdots ,t+l_{d-2}$ and  $t+l_{d-1},$ as $l_1=0$. If $d \geq 3$, then for each pair $i \neq j$, we have the following simple linear relationship between the variables $t+l_i$ and $t+l_j$:
    \begin{equation*}\label{eqn:Kij}
    (l_j-l_i) + (t+l_i) - ( t+l_j)=0.
    \end{equation*}
It follows that for any ${\bf x}=(x_1, \ldots, x_{d-1})$ we have
$$(l_j-l_i)(\Pi_{k=1}^{d-1} (t+l_k)^{x_k}) + (t+l_i)(\Pi_{k=1}^{d-1} (t+l_k)^{x_k}) - ( t+l_j)(\Pi_{k=1}^{d-1} (t+l_k)^{x_k})=0.$$
Thus, we see that in the notation of the presentation for $\ZZZ$,
$$(l_j-l_i)1_{\bf x} + 1_{{\bf x} + {\bf e}_i}- 1_{{\bf x}+{\bf e}_j}\in K.$$ We claim that elements of this form generate $K$.

\begin{proposition}\label{prop:commutatorpresentation}
The kernel $K$ in the exact sequence \begin{equation*}\label{seq:commutator}
1 \rightarrow K \rightarrow \ZZZ \rightarrow \LL \rightarrow 1
\end{equation*} is generated by elements of the form
 \begin{equation*}\label{expn:K-3} (l_j-l_i)1_{\bf x} + 1_{{\bf x} + {\bf e}_i}- 1_{{\bf x}+{\bf e}_j} \end{equation*}
for $1 \leq i \neq j \leq d-1$.
\end{proposition}

\begin{proof}
Let $J$ be the subgroup of $K$ generated by elements of the form $$(l_j-l_i)1_{\bf x}+ 1_{{\bf x} + {\bf e}_i}- 1_{{\bf x}+{\bf e}_j}.$$  We first note that since $\Z_q$ is a cyclic group, and hence $c(1_{\bf x})=c_{\bf x}$ for any $c \in \Z_q$, we obtain  \begin{equation} \label{eqn:Kijplus} (l_j-l_i) c_{\bf x}+ c_{{\bf x} + {\bf e}_i}- c_{{\bf x}+{\bf e}_j} \in J \end{equation}
for arbitrary $c \in \Z_q$. We now show how to rewrite an arbitrary element $(b)_{{\bf v}=(v_1, \ldots, v_{d-1})}$ where $b \in \Z_q$ and $v_j \in \Z$ as a sum of  elements of the form $c _{\bf x}$ , where ${\bf x} \in {\mathbb L}$, and  elements of  $J$, which implies that $J=K$ as desired.

First suppose $v_j  > 0$ for some $j>1$. Then using  an element of $J$ of the form given in \eqref{eqn:Kijplus} with $c=b$, $i=1$, and ${\bf x}={\bf v}-{\bf e}_j$, we have   $ (l_jb)_{{\bf v}-{\bf e}_j}+ b_{{\bf v} + {\bf e}_1-{\bf e}_j}-  b_{\bf v} \in J$, and hence $b_{\bf v}= (l_jb)_{{\bf v}-{\bf e}_j}+  b_{{\bf v} + {\bf e}_1-{\bf e}_j} + \gamma$  for some $\gamma \in J$. Note that the $j^{th}$ coordinate in both of the new vectors ${\bf v}-{\bf e}_j$ and ${\bf v} + {\bf e}_1-{\bf e}_j$ is
 $v_j-1$, and the only other vector coordinate which is affected is the first, so repeated applications of this strategy allow us to rewrite the original element as a sum of elements of the form $c_{{\bf x}}$ with $x_j \leq 0$ for all $j>1$ and elements in $J$.

Thus, we may assume that $v_j \leq 0$ for $j>1$. Next, suppose that $v_i<0$ and $v_j<0$ for some $j>i>1$. Then using the element of the form as in \eqref{eqn:Kijplus} with $c=b(l_j-l_i)^{-1}$ and ${\bf x}={\bf v}$, we have  $b_{\bf v}+(b(l_j-l_i)^{-1}) _{{\bf v}+{\bf e}_i}- (b(l_j-l_i)^{-1}) _{{\bf v} + {\bf e}_j} \in J,$ so $b_{\bf v}= (b(l_j-l_i)^{-1}) _{{\bf v}+{\bf e}_j} -(b(l_j-l_i)^{-1}) _{{\bf v} + {\bf e}_i}+ \gamma$, where $\gamma \in J$.   Note that in the new expression, both
new vector subscripts have either $j^{th}$ coordinate of $v_j+1$ but all other coordinates unchanged, or  the $i^{th}$ coordinate of the subscript is $v_i+1$ and all other coordinates are unchanged, so after repeated applications we can rewrite the original element as  a sum of elements of the form $c_{{\bf x}}$, where $x_k \leq 0$ for all $k>1$ and $x_j < 0$ for at most one index $j>1$, and elements in $J$.

Thus, we may assume that  at most one index $j>1$  in ${\bf v}$ has $v_j<0$; for all other indices $i>1$, $v_i=0$.
 If $v_j=0$ or $v_1=0$, ${\bf v}\in {\mathbb L}$. If $v_j<0$ and $v_1>0$, using the element of $J$ of the form as in \eqref{eqn:Kijplus} with $i=1$, $c=b$ and ${\bf x}={\bf v}-{\bf e}_1$, we see that $b_{\bf v}=b _{{\bf v}+{\bf e}_j-{\bf e}_1} -(l_j b) _{{\bf v} -{\bf e}_1}+ \gamma$  where $\gamma \in J$.  Note that in the right hand expression for $b_{\bf v}$, the vector subscripts of the first two summands have first coordinates whose absolute values are smaller than the absolute value of $v_1$. The same is true for the $j^{th}$ coordinate of the subscript of the first summand and $v_j$, but all other coordinates remain unchanged.

On the other hand, if $v_j<0$ and $v_1<0$, using the above argument with $i=1$, $c=b~ l_j^{-1}$ and ${\bf x}={\bf v}$, we see that $b_{\bf v}=-(b~ l_j^{-1}) _{{\bf v}+{\bf e}_1}+ (b~ l_j^{-1} ) _{{\bf v} + {\bf e}_j}+ \gamma$  where $\gamma \in J$.  Note that in the new expression, the first coordinate of the vector subscript of the first summand has absolute value less than the absolute value of $v_1$. The same is true for the $j^{th}$ coordinate of the subscript of the second summand and $v_j$, but all other coordinates remain unchanged.

Thus, in either case, we see that after repeated applications of the same strategy we may rewrite our original element as a sum of elements of the form $c_{\bf x}$ where ${\bf x} \in {\mathbb L}$  and elements of $J$.
\end{proof}

\begin{proposition}
If $d \geq 3$,
$$\GG =  \langle  a, t_1, t_2, \dots, t_{d-1} | a^q, [a,a^{t_1}], [t_i, t_j] , a^{l_j-l_i}a^{t_i} (a^{t_j})^{-1}  \rangle $$
where the last two relations are included for every $i \neq j$, $1 \leq i,j \leq d-1$.
\end{proposition}

\begin{proof}
We obtain a presentation for $\GG=\GG' \rtimes {\mathbb Z}^{d-1}$ using the semidirect product structure.
We use  multiplicative notation for the group operation, as opposed to the additive notation we used when restricting to the abelian group $\GG'$. For clarity, we  use the formal symbol $a_{\bf x}$, where ${\bf x}=(x_1, \ldots, x_{d-1})$ for the generator $1_{\bf x}$. Then Proposition \ref{prop:commutatorpresentation} yields the following presentation for $\GG'$:
\begin{equation}\label{eqn:Gamma'-pres}
\GG'=\langle  a_{\bf x}, x \in \Z^{d-1} | a_{\bf x}^q,  [a_{\bf x}, a_{\bf y}] , a_{\bf x}^{l_j-l_i} a_{{\bf x}+ {\bf e}_i} ( a_{{\bf x}+{\bf e}_j})^{-1}  \rangle,
\end{equation}
where the relations range over all ${\bf x}$ and all $1 \leq i, j \leq d-1$, $i \neq j$.
Identifying $\GG'$ as the group of matrices of the form $ \left( \begin{array}{cc} 1 & P \\ 0 & 1 \end{array} \right)$ where $P \in {\mathcal R}_d({\mathbb Z}_q)$, the generator $a_{\bf x}$ is identified with the matrix  $\left( \begin{array}{cc} 1 & \Pi_{i=1}^{d-1}(t+l_i)^{x_i} \\ 0 & 1 \end{array} \right)$.

Writing ${\mathbb Z}^{d-1}= \langle t_1, \ldots, t_{d-1}~|~ [t_i,t_j]~ \rangle$, the splitting we choose to express $\GG$ as a semidirect product sends $t_i$ to the  matrix  $\left( \begin{array}{cc} t+l_i & 0 \\ 0 & 1 \end{array} \right)$.
Since $a_{\bf x}^{t_i} = t_iat_i^{-1} = a_{\bf x+e_i}$, we obtain the following presentation for  $\GG=\GG' \rtimes {\mathbb Z}^{d-1}$:
\begin{align*}\GG =  \langle  t_1, t_2, \dots, t_{d-1}, a_{\bf x}, {\bf x} \in \Z^{d-1}~  |~ &a_{\bf x}^q ~,~ [a_{\bf x}, a_{\bf y}]~, ~[t_i, t_j] , \\  & a_{\bf x}^{l_j-l_i}a_{{\bf x}+{\bf e}_i} (a_{{\bf x}+{\bf e}_j})^{-1}, a_{\bf x}^{t_i}=a_{{\bf x}+{\bf e}_i}  \rangle,
\end{align*}
with infinitely many generators and relations.

Relations of the form $ a_{\bf x}^{t_i}=a_{{\bf x}+{\bf e}_i}$ may be used to reduce the generating set to the finite set $ \{  t_1, t_2, \dots, t_{d-1}, a_{\bf 0} \}$, and in the presence of $[t_i, t_j]$, any remaining relations of that form are redundant. In addition, since  $$\left(a_{\bf 0}^{l_j-l_i}a_{{\bf e}_i}(a_{{\bf e}_j})^{-1}\right)^{ \Pi_{k=1}^{d-1}t_k^{x_k}}= a_{\bf x}^{l_j-l_i} a_{{\bf x} + {\bf e}_i} (a_{{\bf x}+{\bf e}_j})^{-1},$$ that infinite subset of relations follows from the finite collection of the form  $a_{\bf 0}^{l_j-l_i}a_{\bf 0}^{t_i} (a_{\bf 0}^{t_j})^{-1}$. Thus, writing $a=a_{\bf 0}$   (note that $a$  then corresponds to the matrix $ \left( \begin{array}{cc} 1 & 1 \\ 0 & 1 \end{array} \right)$) we obtain:
$$\GG =  \langle  t_1, t_2, \dots, t_{d-1}, a ~| ~ a^q, [a^{\Pi_{k=1}^{d-1}t_k^{x_k}}, a^{\Pi_{k=1}^{d-1}t_k^{y_k}}], [t_i, t_j] , a^{l_j-l_i}a^{t_i} (a^{t_j})^{-1} \rangle. $$
Now it is easy to check that the relations $[a^{\Pi_{k=1}^{d-1}t_k^{x_k}}, a^{\Pi_{k=1}^{d-1}t_k^{y_k}}]$ follow from the subset of the form $ [a, a^{\Pi_{k=1}^{d-1}t_k^{x_k}}]$. Furthermore, in the presence of the relations $a^{l_j-l_i} a^{t_i}( a^{t_j})^{-1} $, these all follow from the single relation $ [a, a^{t_1}]$. Reducing the defining relations accordingly yields the desired finite presentation.
\end{proof}

\section {Automorphisms of $\GG$}\label{sec:aut}

We first review the general structure of automorphisms of semidirect products. If $$1 \rightarrow A \rightarrow G \rightarrow B \rightarrow 1$$ is a split exact sequence with $A$ abelian, then it is well known (see, for instance, \cite{Br}) that the group of automorphisms of $G$ which restrict to the identity on $A$ and induce the identity on $B$ is isomorphic to the additive group of derivations from $B$ to $A$, which are defined by
$$Der(B,A)=\{ \delta: B \rightarrow A ~|~ \delta(b_1b_2)=\delta(b_1)\delta( b_2)^{b_1} \},$$ where we denote the action of $B$ on $A$ by conjugation. If $A$ is characteristic, this extends to a characterization of $Aut(G)$; (see \cite{Cu} or \cite{We} for slight variations of this result). We summarize these results in the following proposition.

\begin{proposition}\label{prop:Autsemidirect}
Let $G=A \rtimes B$, with $A$ abelian and characteristic. Then $Aut(G) \cong Der(B,A) \rtimes T$, where $$T= \{ (\alpha, \beta ) \in Aut(A) \times Aut(B) ~|~ \alpha (a^b)=\alpha (a)^{\beta(b)} {\text{ for every}}~ a \in A, b \in B \}.$$
\end{proposition}

\begin{proof}
The action of $T$ on $Der(B,A)$ is defined as follows. If $(\alpha, \beta) \in T$ and $\delta \in Der(B,A)$, then $\delta^{(\alpha, \beta)}=\alpha \circ \delta \circ \beta^{-1}$; one easily checks that this yields a well defined action.
We define a map $f: Aut(G) \rightarrow Der(B,A) \rtimes T$.
Let $f (\vp)=( \delta_{\vp}, (\vp', \Op))$ where we define $\vp', \Op$, and $\delta_{\vp}$ as follows. First, $\vp'$ is defined by $(\vp'(a),1)=\vp(a,1)$ for all $a \in A$. Then  $\Op$, and a third map $\vp''$ which will be used to define $\delta_{\vp}$, are defined by $(  \vp''(b), \Op(b) )=\vp(1,b)$ for all $b \in B$. Then one checks that $\vp(a,b)=(\vp'(a) \vp''(b), \Op(b))$, and that $\vp \in Aut(G)$ implies that $\vp'  \in Aut(A)$ and that $\Op \in Aut(B)$. Using this last fact, we define $\delta_{\vp}(b)=\vp''(\Op^{-1}(b))$.

Since $\vp$ is a homomorphism, $\vp(1,\Op^{-1}(b_1))\vp(1,\Op^{-1}(b_2))=\vp(1, \Op^{-1}(b_1b_2))$, which implies that $\vp''(\Op^{-1}(b_1b_2))=\vp''(\Op^{-1}(b_1))\vp''(\Op^{-1}(b_2))^{b_1}$, so
 $\delta_{\vp} \in Der(B,A)$.   Finally, since $\vp(1,b)\vp(a,1)=\vp(a^b,b)$ for every $a \in A$, $b \in B$, we see that $\vp'(a^b)=\vp'(a)^{\Op(b)}$, so $(\vp', \Op) \in T$, and the map $f$ is well defined.
One then verifies that $f$ is a group homomorphism.

This identification can be reversed: given an ordered pair $(\delta, (\alpha, \beta))$ with $\delta \in Der(B,A)$, $(\alpha, \beta) \in T$,  we set $g((\delta, (\alpha, \beta)))(a,b)=(\alpha(a) \delta(\beta(b)), \beta(b))$, and that the fact that $(\alpha, \beta) \in T$ is sufficient to show that $g((\delta, (\alpha, \beta)) \in Aut(G)$. Moreover, one checks that $g \circ f$ and $f \circ g$ are the identity on their respective domains.
\end{proof}

Before applying Proposition \ref{prop:Autsemidirect} to $\GG = \GG' \rtimes \Z^{d-1} \cong {\mathcal R}_d(\Z _q) \rtimes \Z^{d-1}$ to obtain a characterization of $Aut(\GG)$, we establish some notation. Recall the following exact sequence:
\begin{equation*}
1 \rightarrow K \rightarrow \ZZZ \rightarrow \LL \rightarrow 1,
\end{equation*}
which depicts $\GG'=\LL$ as a quotient of $\ZZZ$. We showed that $K$ is generated by elements of the form $(l_j-l_i) 1_{\bf x}+ 1_{{\bf x}+ {\bf e}_i}-  1_{{\bf x}+{\bf e}_j}$ for $i \neq j$, $1 \leq i,j \leq d-1$, so that the presentation in Equation \eqref{eqn:Gamma'-pres} for $\GG'$ can be rewritten as:

$$\GG'=\langle  1_{\bf x}, x \in \Z^{d-1} |q 1_{\bf x},  [1_{\bf x}, 1_{\bf y}] ,(l_j-l_i) 1_{\bf x}+ 1_{{\bf x}+ {\bf e}_i}-  1_{{\bf x}+{\bf e}_j}  \rangle.$$

We use additive notation for the group operation since both ${\mathcal R}_d(\Z_q)$ and $\ZZZ$ are rings, however we view the group operation as multiplication of matrices when elements are expressed in that form, and as addition when identifying $\GG'$ with ${\mathcal R}_d(\Z_q)$ and viewing the elements of $\GG'$ as polynomials.
In the ring $\ZZZ$, multiplication is defined via $1_{\bf v}1_{\bf w}=1_{{\bf v}+{\bf w}}$, and extended to make multiplication distribute over addition. It is easy to see that $K$ is an ideal, for if  $k= (l_j-l_i)1_{\bf x} + 1_{{\bf x} - {\bf e}_i}- 1_{{\bf x}+{\bf e}_j}$ is a generator for $K$ as an additive group, then \begin{align*}1_{\bf v} k&= 1_{\bf v}( (l_j-l_i)1_{\bf x} + 1_{{\bf x} +{\bf e}_i}- 1_{{\bf x}+{\bf e}_j}) \\ &=  (l_j-l_i)1_{({\bf x}+{\bf v})} + 1_{({\bf x}+{\bf v}) + {\bf e}_i}- 1_{({\bf x}+{\bf v})+{\bf e}_j} \end{align*} which is itself a generator of $K$. Hence, $\GG'$ is isomorphic, as a ring, to $(\ZZZ) / K$. Moreover, since
$$ 1_{\bf x}( (l_j-l_i)1_{\bf 0} + 1_{ {\bf e}_i} - 1_{{\bf e}_j}) =  (l_j-l_i)1_{\bf x} + 1_{{\bf x} + {\bf e}_i}- 1_{{\bf x}+{\bf e}_j},$$  $K$ is finitely generated as an ideal by the set $$\{  (l_j-l_i)1_{\bf 0} + 1_{ {\bf e}_i}- 1_{{\bf e}_j}~|~ i \neq j, 1 \leq i,j \leq d-1\}.$$
Now $\Z^{d-1}$ acts on $\ZZZ$ via  $$1_{\bf x}^{\bf v}=1_{{\bf x}+{\bf v}}=1_{\bf x}1_{\bf v}.$$
Since this is an action by multiplication in this ring, and $K$ is an ideal, it induces an action of  $\Z^{d-1}$ on $\GG'={\mathcal R}_d(\Z_q)$.

In addition,  $Aut(\Z^{d-1})$ acts on $\ZZZ $. To see this, for $\beta \in Aut(\Z^{d-1})$ define $1_{\bf v}^{\beta}=  1_{\beta({\bf v})}$, and extend via $(\sum_{i=1}^r 1_{{\bf v}_i})^{\beta}= \sum_{i=1}^r (1_{{\bf v}_i})^{\beta}$.  Note that since $1_{\bf v}1_{\bf w}=1_{{\bf v}+{\bf w}}$, we have $(1_{\bf v}1_{\bf w})^{\beta}=1_{\bf v}^{\beta} 1_{\bf w}^{\beta}$, so $Aut(\Z^{d-1})$ acts on $\ZZZ$ via ring homomorphisms. The two actions on $\ZZZ$ interact as follows. For every $A \in \ZZZ, {\bf v} \in \Z^{d-1}, \beta \in Aut(\Z^{d-1})$ we have
$$(A^{\bf v})^{\beta}=(1_{\bf v} A)^{\beta}=1_{\bf v}^{\beta}A^{\beta}=1_{\beta({\bf v})}A^{\beta}=(A^{\beta})^{\beta({\bf v})},$$ so $$(A^{\bf v})^{\beta}= (A^{\beta})^{\beta({\bf v})}.$$

It is not always true that the action of  $Aut(\Z^{d-1})$ on $\ZZZ$ induces an action on the quotient ring $\LL=\GG'$. However, for fixed $\beta \in Aut(\Z^{d-1})$, if $k^{\beta} \in K$ for every $k \in K$, then the action by the element $\beta$ does pass to the quotient. Define $\p= \{ \beta \in Aut(\Z^{d-1}) | K^{\beta}=K \}$. If $\beta \in \p$, then $\beta^{-1} \in \p$, $\p$ is clearly a subgroup of $Aut(\Z^{d-1})$, and $\p$ does act on the quotient $\GG'$ via ring homomorphisms.

Before stating the main theorem, recall that for any ring $R$, the multiplicative subgroup of the units in $R$ is denoted $U(R)$.  We now compute the automorphism group of $\GG$.

\begin{theorem}\label{thm:AutGamma}
For any $d \geq 2$, $$Aut(\GG) \cong Der(\Z^{d-1}, {\mathcal R}_d(\Z_q) ) \rtimes (U({\mathcal R}_d(\Z_q)) \rtimes \p),$$ where
$${\p}= \{ \beta  \in Aut(\Z^{d-1}) | K^{\beta}=K \}.$$
\end{theorem}

Before beginning the proof, we establish some notational conventions. We denote elements of $\Z^{d-1}$ by vectors ${\bf v}=(v_1, \ldots, v_{d-1})$, $v_i \in \Z$, where the generator $t_i$ of $\GG$ corresponds to the standard basis vector ${\bf e}_i$. Furthermore, we denote elements of $\GG'$ by polynomials in ${\mathcal R}_d(\Z_q)$. An element $\beta \in Aut(\Z^{d-1})$ can be represented by a $(d-1) \times (d-1)$ matrix with entries in $\Z$ with respect to the standard basis, and we freely identify $\beta$ with its matrix $(b_{i,j})$.

\begin{proof}According to Proposition \ref{prop:Autsemidirect}, $Aut(\GG) \cong Der(\Z^{d-1}, {\mathcal R}_d(\Z_q) ) \rtimes T$, where
$$T =\{ (\alpha, \beta) \in Aut({\mathcal R}_d(\Z_q) )  \times Aut(\Z^{d-1})~|~  \alpha(g^{\bf v})=\alpha(g)^{\beta({\bf v})}   \},$$ and the condition on $(\alpha, \beta)$ holds for every $ g \in \GG', {\bf v} \in \Z^{d-1}$.

Note that the action of $\p$ on $\GG'={\mathcal R}_d(\Z_q)$ preserves $U({\mathcal R}_d(\Z_q))$, for if $R \in U({\mathcal R}_d(\Z_q)$, then $ R^{\beta}( R^{-1})^{\beta}=(R R^{-1})^{\beta}=(1)^{\beta}=1$. Thus, $\p$ acts on the multiplicative group $ U({\mathcal R}_d(\Z_q)) $, and we use this action to construct the semidirect product  $U({\mathcal R}_d(\Z_q)) \rtimes \p$. In other words, multiplication is given by $(R, \beta)(S , \gamma)=(RS^{\beta}, \beta \circ\gamma)$.

To complete the proof of the theorem, we must show that $T \cong U({\mathcal R}_d(\Z_q)) \rtimes \p$. We first define a map $f$ from $T$ to $U({\mathcal R}_d(\Z_q)) \rtimes \p$ via $f(\alpha, \beta)=(\alpha(1), \beta)$ for $(\alpha, \beta) \in T$. We claim that $\alpha(1) \in U({\mathcal R}_d(\Z_q)) $ and $\beta \in \p$, so the map $f$ is well defined.

The fact that that $\alpha(1) \in U({\mathcal R}_d(\Z_q)) $ follows from the surjectivity of $\alpha$. Since $(\alpha, \beta) \in T$, we know that $\alpha(Q^{\bf v})=\alpha(Q)^{\beta({\bf v})}$ for every $Q \in {\mathcal R}_d(\Z_q)$ and ${\bf v} \in \Z^{d-1}$, and
 since $\alpha$ is surjective, we have $\alpha(P)=1$ for some $P \in {\mathcal R}_d(\Z_q)$. In addition, observe that $Q^{\bf v}=Q1^{\bf v} $ for any $Q \in {\mathcal R}_d(\Z_q)$ and ${\bf v} \in \Z^{d-1}$.
Express $P$ as a sum $P=\sum_{k=1}^r 1^{{\bf v}_k}$. Then we compute.
\begin{align*} \alpha(P)&=\alpha(\sum_{k=1}^r 1^{{\bf v}_k})=\sum_{k=1}^r  \alpha(1^{{\bf v}_k})=
 \sum_{k=1}^r \alpha(1)^{\beta({\bf v}_k)} \\ &=\sum_{k=1}^r \alpha(1) 1^{\beta({\bf v}_k)} = \alpha(1) \left( \sum_{k=1}^r 1^{\beta({\bf v}_k)}\right)=1.
\end{align*}
Thus  $\alpha(1)S=1$ where $S= \sum_{k=1}^r 1^{\beta({\bf v}_k)}$, so $\alpha(1)  \in U({\mathcal R}_d(\Z_q)) $ .

Next we claim that $\beta \in \p$. Since $Aut(\Z^{d-1})$ acts on $\ZZZ$ via ring homomorphisms, it suffices to show that $k^{\beta}$ and $k^{\beta^{-1}}$ are both elements of $K$ for every $k$ in the finite set of generators for the ideal $K$. Let $$k=(l_j-l_i)1_{\bf 0}+ 1_{{\bf e}_i}-1_{{\bf e}_j}$$ be one such generator.  We must show that $$k^{\beta}=  (l_j-l_i) 1_{\bf 0}+ 1_{\beta({\bf e}_i)}-1_{\beta({\bf e}_j)}=  (l_j-l_i) 1_{\bf 0}+ 1_{\bf 0}^{\beta({\bf e}_i)}-1_{\bf 0}^{\beta({\bf e}_j)}\in K.$$
Using the facts that $\alpha(Q^{\bf v})=\alpha(Q)^{\beta({\bf v})}$ and $Q^{\bf v}=Q ( 1^{\bf v}) $ for any $Q \in {\mathcal R}_d(\Z_q)$, ${\bf v} \in \Z^{d-1}$, we compute:
\begin{align*}
\alpha(0) &=\alpha((l_j-l_i)+ (t+l_i)-(t+l_j)) \\ &=\alpha((l_j-l_i)+ 1^{{\bf e}_i}-1^{{\bf e}_j}) \\ &=
(l_j-l_i)\alpha(1)+ \alpha(1^{{\bf e}_i})-\alpha(1^{{\bf e}_j}) \\
&= (l_j-l_i)\alpha(1)+ \alpha(1)^{\beta({\bf e}_i)}-\alpha(1)^{\beta({\bf e}_j)} \\
&= (l_j-l_i)\alpha(1)+  \alpha(1) (1^{\beta({\bf e}_i)})- \alpha(1)( 1^{\beta({\bf e}_j)})\\
&=\alpha(1)( (l_j-l_i)+  1^{\beta({\bf e}_i)}- 1^{\beta({\bf e}_j)})=0.
 \end{align*}
Since $\alpha(1)$ is invertible, this implies that $(l_j-l_i)+  1^{\beta({\bf e}_i)}- 1^{\beta({\bf e}_j)}=0$ in ${\mathcal R}_d(\Z_q)$, which in turn implies that  $k^{\beta}=  (l_j-l_i) 1_{\bf 0}+ 1_{\bf 0}^{\beta({\bf e}_i)}-1_{\bf 0}^{\beta({\bf e}_j)}\in K$, as desired.
Now $(\alpha, \beta) \in T$ implies $(\alpha^{-1}, \beta^{-1}) \in T$, so  a similar argument shows that $k^{\beta^{-1}} \in K$, and hence $\beta \in \p$, and the map $f$ is well defined.

We remark that since $\beta \in \p$, the action of $\beta$ on $\ZZZ$ induces an action on $\GG'$, which yields  a simple formula for $\alpha$. To obtain the formula, recall that in an earlier computation, we showed that if $S\in {\mathcal R}_q(\Z_d)$ is written as $S=\sum_{k=1}^r 1^{{\bf v}_k}$,  then
\begin{equation}\label{eqn:alpha}
\alpha(S)= \alpha(1) \left( \sum_{k=1}^r 1^{\beta({\bf v}_k)}\right).
\end{equation}
But since $(1^{\bf v})^{\beta}=(1^{\beta})^{\beta({\bf v})}=1^{\beta({\bf v})}$ for any ${\bf v}\in \Z^{d-1}$,
the formula for $\alpha$ simplifies:
$$\alpha(S)=\alpha(1)\sum_{k=1}^r 1^{\beta({\bf v}_k)}=\alpha(1)\sum_{k=1}^r (1^{{\bf v}_k})^{\beta}= \alpha(1)\left( \sum_{k=1}^r 1^{{\bf v}_k} \right)^{\beta}=\alpha(1) S^{\beta}$$
In particular, if $(\alpha_1, \beta_1) , (\alpha_2, \beta_2) \in T$, then it follows from Equation \eqref{eqn:alpha} that $\alpha_1(\alpha_2(1))=\alpha_1(1) \alpha_2(1)^{\beta_1}$, so
\begin{align*}f(\alpha_1, \beta_1) f(\alpha_2, \beta_2) &=(\alpha_1(1), \beta_1)(\alpha_2(1), \beta_2)= (\alpha_1(1)(\alpha_2(1))^{\beta_1}, \beta_1 \beta_2)\\
&= (\alpha_1(\alpha_2(1)), \beta_1\beta_2) = f(\alpha_1\alpha_2, \beta_1\beta_2)\\ &=f((\alpha_1, \beta_1)(\alpha_2, \beta_2)),
 \end{align*}
and hence $f$ is a group homomorphism.

Now define a map $g$ from $U({\mathcal R}_d(\Z_q)) \rtimes \p$ to $\mathcal T$ via $g(R, \beta)=(\alpha_{R, \beta}, \beta)$, where we define $\alpha_{R, \beta}(S)=RS^{\beta}$ for any $S \in \GG'$. One easily verifies that $\alpha_{R, \beta}$ is a group homomorphism, for $\alpha_{R, \beta}(S_1+S_2)=R(S_1+S_2)^{\beta}=R(S_1^{\beta}+S_2^{\beta})=RS_1^{\beta}+ RS_2^{\beta}=\alpha_{R, \beta}(S_1)+\alpha_{R, \beta}(S_2)$. Since $R \in U({\mathcal R}_d(\Z_q)) $, we know that $R^{-1} \in U({\mathcal R}_d(\Z_q))$, and since $\beta \in \p$, we know that $\beta^{-1} \in \p$ as well, which in turn implies that  $(R^{-1})^{\beta^{-1}} \in  U({\mathcal R}_d(\Z_q))$. Thus, $( (R^{-1})^{\beta^{-1}}, \beta^{-1}) \in  U({\mathcal R}_d(\Z_q)) \rtimes \p$. One easily checks that $\alpha_{(R^{-1})^{\beta^{-1}}, \beta^{-1}} ( \alpha_{R, \beta} (S))= \alpha_{R, \beta} ( \alpha_{(R^{-1})^{\beta^{-1}}, \beta^{-1}} (S))=S$ for any $S \in {\mathcal R}_d(\Z_q)$, and hence $\alpha_{R, \beta} \in Aut(\GG')$.
 To see that $(\alpha_{R, \beta}, \beta) \in T$, let $P \in \GG'$ and ${\bf v} \in \Z^{d-1}$. Then
$\alpha_{R, \beta}(P^{\bf v})=R (P^{\bf v})^{\beta}=R (P^{\beta})^{\beta({\bf v})}= \alpha_{R, \beta}(P)^{\beta({\bf v})},$ and hence $(\alpha_{R, \beta}, \beta) \in T$. Thus the map $g$ is well defined.

It is easily verified that $f(g(R, \beta))=f(\alpha_{R, \beta}, \beta) =(R, \beta)$ and $g(f(\alpha, \beta))=g(\alpha(1), \beta)=(\alpha_{\alpha(1), \beta}, \beta)$. But $\alpha_{\alpha(1), \beta}(S)=\alpha(1) S^{\beta}$ for all $S \in \GG'$, so $\alpha_{\alpha(1), \beta}=\alpha$, and  $g(f(\alpha, \beta))=(\alpha, \beta)$. Thus $f$ is a group isomorphism, and therefore
$T \cong U({\mathcal R}_d(\Z_q)) \rtimes \p$, as desired.
\end{proof}

In the next two subsections, we characterize the subgroup $\p$ and determine the outer automorphism group.

\subsection{Characterizing $\p$}

Recall that  ${\p}= \{ \beta  \in Aut(\Z^{d-1}) | K^{\beta}=K \}$, where $K$ is the kernel of the the natural surjection from $\ZZZ$ to $\LL$. In the case $d=2$, $K=\{0\}$, so $\p=Aut(\Z)\cong\Z_2$ and Theorem \ref{thm:AutGamma} simplifies as follows.

\begin{theorem}
$Aut(L_q)\cong Der(\Z, {\mathcal R}_2(\Z_q) ) \rtimes (U({\mathcal R}_2(\Z_q)) \rtimes \Z_2)$.
\end{theorem}

If $d \geq 3$, even though $Aut(\Z^{d-1})$ contains more possible automorphisms, very few of them arise as elements of $\p$.
 Let $(b_{i,j})$ be the matrix for $\beta$ and let  $(c_{i,j})$ be the matrix for $\beta^{-1}$. Let
\begin{align*} S = \{ \beta \in Aut(\Z^{d-1}) ~|~& (l_j-l_i)+ \Pi_{k=1}^{d-1} (t+l_k)^{b_{k,i}}- \Pi_{m=1}^{d-1} (t+l_m)^{b_{m,j}}=0 ~ \forall ~ i \neq j ,\\
 & (l_j-l_i)+ \Pi_{k=1}^{d-1} (t+l_k)^{c_{k,i}}- \Pi_{m=1}^{d-1} (t+l_m)^{c_{m,j}}=0 ~ \forall ~ i \neq j \}. \end{align*}
We claim that $S = \p$.
To see this, note that $(l_j-l_i)1_{\bf 0}+ 1_{\beta({\bf e}_i)}-1_{\beta({\bf e}_j)} \in K$ if and only if
$(l_j-l_i)+ \Pi_{k=1}^{d-1} (t+l_k)^{b_{k,i}}- \Pi_{m=1}^{d-1} (t+l_m)^{b_{m,j}}=0$, so $\beta \in S$ if and only if $k^{\beta}$ and $k^{\beta^{-1}}$ are in $K$ for each generator $k$ of $K$. But this is equivalent to the condition $K^{\beta}=K$, or $\beta \in \p$. So $\beta \in S$ if and only if $\beta \in \p$.

We now use these conditions on the matrix entries to show that $\p$ is quite restricted when $d \geq 3$.

\begin{theorem}\label{thm:S}
Let $\p$ be as above with $d \geq 3$. Then we have:
\begin{enumerate}
\item If $d>3$ and $q=d-1$ is prime, then $ \p \cong\Z^{d-1}$. Furthermore, without loss of generality if $l_i=i-1$ for $1 \leq i \leq d-1$, then $\p$ is generated by the permutation $\beta$ with matrix $(b_{i,j})$ where $b_{i+1,i}=b_{1,d}=1$ and all other entries are 0.
\item If $d=3$ and $q>2$ and $l_2 \in \{ \pm 1\}$, then $\p \cong\Z^2$.
If $l_2=-1$, then $\p$ is generated by $\beta=\left( \begin{array}{cc} 1 & 0 \\ -1 & -1 \end{array} \right)$, whereas
if $l_2=1$, then $\p$ is generated by $\beta=\left( \begin{array}{cc} -1 & -1 \\ 0 & 1 \end{array} \right)$.
\item If $d=3$ and $q=2$, then $\p \cong D_3$ In this case, in addition to the identity, $\p$ contains the two matrices in case (2) above, as well as the  matrices $\left( \begin{array}{cc} -1 & -1  \\ 1 & 0 \end{array} \right), \ \left( \begin{array}{cc} 0 & 1 \\ -1 & -1 \end{array} \right),  \text{ and } \left( \begin{array}{cc} 0 & 1 \\ 1 & 0 \end{array} \right).$
\item In all other cases, $\p$ is trivial.
\end{enumerate}
\end{theorem}

Before embarking on the proof of this theorem we prove some lemmas restricting the entries of $(b_{i,j})$.

\begin{lemma}\label{prop:columnsum} Let $\beta \in \p$ have matrix  $ (b_{i,j})$, and let $C(i)= \sum_{k=1}^{d-1} b_{k,i}$, the sum of the entries in the $i^{th}$ column of $(b_{i,j})$.  Then the following hold:
\begin{enumerate}
\item $C(i)<0$ for at most one $i$ with $1 \leq i \leq d-1$,
\item $C(i)=0$ for at most one $i$ with $1 \leq i \leq d-1$, and
\item if $C(i)>0$ for some $i$, then $C(j)=C(i)$ for all $j$ with $1 \leq j \leq d-1$.
\end{enumerate}
\end{lemma}

\begin{proof}
Since $\beta \in \p$, the following equations involving the entries in the matrix $(b_{i,j})$ for $\beta$ hold for all $i \neq j$:
\begin{equation} \label{eqn:Qij} Q_{i,j}=(l_j-l_i)+ \underbrace{\Pi_{k=1}^{d-1} (t+l_k)^{b_{k,i}}}_{S_i}- \underbrace{\Pi_{m=1}^{d-1} (t+l_m)^{b_{m,j}}}_{S_j}=0.
\end{equation}
Recall that the Decomposition Lemma  (Lemma \ref{lemma:decomposition})  gives a unique decomposition of $Q_{i,j}$ into $d$ polynomials $P_k(Q_{i,j})$ for $1 \leq k \leq d$. Since $Q_{i,j}=0$, we must have $P_k(Q_{i,j}) = 0$ for $1 \leq k \leq  d$. To prove the lemma we compute the formal Laurent series ${\mathcal LS}_d(\Qi)$ in the variable $t^{-1}$ by computing ${\mathcal LS}_d(l_j-l_i)+{\mathcal LS}_d(S_i)-{\mathcal LS}_d(S_j)$. Then $P_d(Q_{i,j})$ consists of the sum of all terms of non-positive degree in this sum, so setting $P_d(Q_{i,j})=0$ yields restrictions on the matrix entries.

First,  ${\mathcal LS}_d(l_j-l_i)=l_j-l_i\neq 0$, since $l_j-l_i$ is invertible.  We claim that  ${\mathcal LS}_d(S_i)$ has lowest degree term with degree $-C(i)=-\sum_{k=1}^{d-1}b_{k,i}$.  To see this, note that:

\begin{itemize}
\item If $b_{k,i} \geq 0$, then $(t+l_k)^{b_{k,i}}$ can be expanded easily as a polynomial in $t$, and then rewritten as $(\ti)^{-b_{k,i}}$+ (terms of higher degree).
\item If $b_{k,i} <0$ and $k=1$, then as $l_1=0$ we have $(t+l_k)^{b_{k,i}}=(t^{-1})^{-b_{k,i}}$.

\item If $b_{k,i} <0$ and $k \neq 1$, we expand $$\left((t+l_k)^{-1}\right)^{-b_{k,i}}=\left(\frac{t^{-1}}{1+l_kt^{-1}}\right)^{-b_{k,i}} = \left(\sum_{v \geq 1} c_v (\ti)^v \right)^{-b_{k,i}} = \sum_{v \geq -b_{k,i}} c_v' (\ti)^v$$
where $c_v,c_v' \in \Z_q$ and the initial coefficients $c_1$ and $c_{-b_{k,i}}'$ both equal $1$.

\end{itemize}
Substituting these expressions into $S_i=\Pi_{k=1}^{d-1} (t+l_k)^{b_{k,i}}$ and expanding, we see that the lowest degree term in ${\mathcal LS}_d(S_i)$ has degree $-C(i)=- \sum_{k=1}^{d-1} b_{k,i}$ and has coefficient $1$.  Similarly, we  see that the lowest degree term in ${\mathcal LS}_d(S_j)$ has degree $-C(j)$ and coefficient 1.

If $C(i)<0$ and $C(j)<0$ for some distinct values of $i$ and $j$, then $${\mathcal LS}_d(\Qi)= (l_j-l_i)+ ( \text{terms with positive degrees}),$$ and hence $P_d(Q_{i,j})=l_j-l_i \neq 0$, a contradiction.  Therefore at most one column of $(b_{i,j})$ can have negative sum.  Similarly, if for some distinct $i$ and $j$ we have $C(i)=C(j)=0$ then as the both ${\mathcal L}_d(S_i)$ and ${\mathcal L}_d(S_j)$ have constant term $1$ and no terms with negative degree, $P_d(Q_{i,j})=l_j-l_i \neq 0$, a contradiction. Hence at most one column has sum equal to zero.

Now suppose for some $i$ we have $C(i) >0$.  If $C(j) \neq C(i)$ for some $j \neq i$, then the analysis above shows that $P_d(\Qi)$ contains a term with strictly negative exponent, a contradiction since $P_d(Q_{i,j}) = 0$.  Thus if $C(i) >0$ for some $i$ we must have $C(i) = C(j)$ for all $j \neq i$, in which case the terms of degree $-C(i)$ cancel when we compute $P_d(\Qi)$.
\end{proof}

\begin{lemma}\label{lem:negativeentries}
Let $\beta \in \p$ have matrix $(b_{i,j})$. If $b_{n,i}<0$ for some $ 1 \leq n,i \leq d-1$, then for every $j$, $b_{n,j}=-1$  and $b_{k,j} \geq 0$ if $k \neq n$.
\end{lemma}

\begin{proof}
Suppose  $b_{n,i}<0$, choose $j \neq i$, and recall that $P_n(Q_{i,j})=0$. As before, we compute ${\mathcal LS}_n(Q_{i,j})$ by computing the Laurent series separately for $l_j-l_i$, $S_i$ and $S_j$. First,  ${\mathcal LS}_n(l_j-l_i)=l_j-l_i$, and since $n<d$, this contributes no terms to $P_n(Q_{i,j})$. If $k \neq n$, then  $(t+l_k)^{b_{k,i}} = (l_k-l_n)^{b_{k,i}}$ plus terms in $t+l_n$ of higher degree, so ${\mathcal LS}_n(S_i)$ has lowest degree term  $c_i(t+l_n)^{b_{n,i}}$ for some invertible $c_i \in \Z_q$.
Similarly the minimal degree term in ${\mathcal LS}_n(S_j)$ is $c_j(t+l_n)^{b_{n,j}}$ for some invertible $c_j \in \Z_q$.  Since $b_{n,i} <0$ we see unless $b_{n,i}=b_{n,j}$ we will not have $P_n(\Qi) = 0$ as required.  By varying $j$ we conclude that all entries of the $n$-th row of  $\beta$ are identical and negative, say with value $s$.  But then $s$ divides $det(\beta) = \pm 1$, and since $s<0$ we conclude that $s=-1$. If $b_{k,j} <0$ for $k \neq n$, then rows $k$ and $n$ are identical, contradicting the fact that the matrix is invertible. Thus, $b_{k,j} \geq 0$ if $k \neq n$.
\end{proof}

\begin{proof}[Proof of Theorem  \ref{thm:S}]
If $d>3$, Lemma \ref{prop:columnsum} shows that all column sums of  $(b_{i,j})$, the matrix for $\beta$,  must be positive and equal.  Let  $C(i)=s>0$ for every $1 \leq i \leq d-1$. Adding every row to the last row yields a matrix whose determinant is still $det(b_{i,j}) = \pm 1$, whose final row has all entries equal to $s$. Since $s$ divides $det(b_{i,j})$, we must have $s=1$.
We claim that in fact, all entries of  $(b_{i,j})$ are non-negative.
Suppose to the contrary that some entry of  $(b_{i,j})$ is negative, say $b_{n,i}<0$.  By Lemma \ref{lem:negativeentries} we know that all entries in row $n$ are $-1$.  Now choose a row $k \neq n$, and add all rows of $(b_{i,j})$ to its $k$-th row.  Rows $n$ and $k$ of the resulting matrix are linearly dependent, contradicting the fact that $det(b_{i,j}) = \pm 1$.
 Hence $b_{i,j} \geq 0$ for all $i$ and $j$, and it follows that $(b_{i,j})$ is a permutation matrix.

Suppose that $\beta$  corresponds to the permutation $\sigma \in \Sigma_{d-1}$. Thus, $b_{k,j}=1$ if $k=\sigma(j)$ and $b_{k,j}=0$ if $k \neq \sigma(j)$, so Equation \eqref{eqn:Qij}
becomes simply $Q_{i,j}=l_j-l_i+ (t+l_{\sigma(i)}) - (t+l_{\sigma(j)}) = 0$,
which simplifies to an equation in $\Z_q$, namely:
$$Q_{i,j}= l_j-l_i+ l_{\sigma(i)} -l_{\sigma(j)} = 0$$
Write $\sigma = \sigma_1 \sigma_2 \cdots \sigma_r$ as a product of disjoint cycles.  Suppose this decomposition contains a $k$-cycle for $2 \leq k \leq d-1 \leq q$, say  $\sigma_1=(i_1 \cdots i_k)$. We claim that $k(l_{i_1}-l_{i_2})=0$, but since $(l_{i_1}-l_{i_2})$ is invertible, this implies that $k=q=d-1$. Since $d-1 \leq p$ for any prime $p$ dividing $q$, it follows that $q$ is prime.

To verify this claim, we first examine the equation  $Q_{i_1, i_k}= l_{i_k}-l_{i_1}+l_{i_2}-l_{i_1}=0$, which implies that $l_{i_k}=l_{i_1}+(l_{i_1}-l_{i_2})$. If $k=2$, this yields $2( l_{i_2}-l_{i_1}) =0$ as desired.
If $k \geq 3$, we use the $k-2$  equations for $Q_{i_1, i_2}, \cdots ,Q_{i_{1}, i_{ k-1}}$ inductively to show that $l_{i_m}=l_{i_1}+(m-1) (l_{i_2}-l_{i_1})$ for $3 \leq m \leq k$. Combining $l_{i_k}=l_{i_1}+(k-1) (l_{i_2}-l_{i_1})$ with $l_{i_k}=l_{i_1}+(l_{i_1}-l_{i_2})$ yields $k(l_{i_1}-l_{i_2})=0$, as desired, and hence $k=q=d-1$.

Now since $k=q=d-1$, $\sigma$ is a single cycle. In addition, $\{ l_1, \dots, l_{d-1} \}=\Z_q$. Without loss of generality, let $l_{i_1}=l_1=0$, and suppose $l_2=1, l_3=2, \dots, l_{d-1}=d-2$. Then $1=l_2=l_{i_j}$ for some $j$, and $\sigma= (1 2
\cdots d-1 )^{-(j-1)}$. Moreover, it is easily verified that if $\gamma \in Aut(\Z^{d-1})$ corresponds to a permutation matrix, then $\gamma \in \p$.

If d=3 and one column sum of $\beta$ is positive, the argument above shows that $\beta$ must be a permutation matrix. But the second part of the argument also shows that if $\beta$ is a transposition, then $q=d-1=2$. Therefore, $\left( \begin{array}{cc} 0 & 1 \\ 1 & 0 \end{array} \right)$ only occurs if $q=2$.

If $d=3$ and neither column sum is positive, it follows from Lemma \ref{prop:columnsum} that one column sum is negative and the other column sum is zero. If $C(i)=0$, and the entries in column $i$ are $n$ and $-n$, it follows immediately that $n$ divides the determinant and hence $n = \pm 1$.  But then since one entry is $-1$,  by Lemma \ref{lem:negativeentries}, the other entry in that row is also $-1$, and if $x$ is the remaining entry, then $x\geq 0$.  But $C(j)=x-1<0$, so $x<1$, and therefore $x=0$, and $(b_{i,j})$ is one of the following four matrices, two matrices with determinant $-1$, namely $\beta_1:\left( \begin{array}{cc} 1 & 0 \\ -1 & -1 \end{array} \right)$ and $\beta_2:\left( \begin{array}{cc} -1 & -1 \\ 0 & 1 \end{array} \right)$, and two with determinant $1$, $\beta_3:\left( \begin{array}{cc} -1 & -1 \\ 1 & 0 \end{array} \right)$ and $\beta_4:\left( \begin{array}{cc} 0 & 1 \\ -1 & -1 \end{array} \right)$.

In each case, substituting all four matrix entries into Equation \eqref{eqn:Qij} yields further restrictions on $l_2$, especially in the cases where the matrix has determinant $1$.  For $\beta_1$ we obtain:
$$Q_{1,2}= l_2+(t)(t+l_2)^{-1} - (t+l_2)^{-1}=0,$$
where we have substituted $l_1=0$. Now we simplify, multiplying both sides by $t+l_2$, to obtain $(l_2+1)t+(l_2^2-1)=0$. Thus, we conclude that this matrix can only arise if $l_2=-1$. Similarly, for $\beta_2$ we obtain $(l_2-1)t+(1-l_2)=0$,  which shows that this matrix only arises if $l_2=1$. Therefore, if $q>2$, since $1 \neq -1$,  at most one of $\beta_1$ or $\beta_2$ arises, depending on the choice of $l_2$.

For $\beta_3$ we have $(l_2+1)t+(l_2-1)=0$, and hence $l_2=-1=1$, which implies that $q=2$.
Similarly, for $\beta_4$, we have $(1-l_2)t+(-l_2^2-1)=0$. Thus, $l_2=1$ and $2=0$, which also implies that $q=2$. Thus, if $q=2$, all four matrices, as well as the transposition, arise. This completes the proof of Theorem \ref{thm:S}.
\end{proof}

Note that for a fixed $d$, choosing $q=d-1$ yields the ``smallest" example of a Diestel-Leader group whose Cayley graph $\Gamma(\GG,S_{d,q})$ is $DL_d(q)$, meaning that $q=d-1$ is the minimal value for which the construction in \cite{BNW} holds.

\subsection{Characterizing $Out(\GG)$.}

\begin{theorem}\label{thm:out}
If $d \geq 3$, $$Out(\GG) \cong ( U({\mathcal R}_d(\Z_q))/ M) \rtimes \p$$  and if $d=2$,
$$Out(\Gamma_2(q)) \cong (\Z_q[t, t^{-1}]/ \langle t-1  \rangle ) \rtimes ( (U(\Z_q[t,t^{-1}])/M) \rtimes \Z_2),$$

where $M= \{ \Pi_{i=1}^{d-1} (t+l_i)^{x_i} | x_i \in \Z \}$ is the set of monomials with coefficient one.
\end{theorem}

\begin{proof}
It is easy to characterize $Inn(\GG)$, the group of inner automorphisms. If $g= \left( \begin{array}{cc} \Pi_{i=1}^{d-1} (t+l_i)^{x_i} & P \\ 0 & 1 \end{array} \right) \in \GG$, then $\phi_g \in Aut(\GG)$, the automorphism given by $\phi_g(h)=g h g^{-1}$ corresponds to the element $$(\delta_P, (\Pi_{i=1}^{d-1} (t+l_i)^{x_i}, id)) \in Der(\Z^{d-1}, {\mathcal R}_d(\Z_q)) \rtimes ( U({\mathcal R}_d(\Z_q)) \rtimes \p)),$$ where $\delta_P({\bf v})=P^{\bf v}-P$.  For any group $G$ and $G$-module $A$, the principal derivations from $G$ to $A$, denoted $P(G,A)$, is the subgroup of all $\delta \in Der(G,A)$ for which there exists some $m \in A$ such that  $\delta(g)=m^g-m$ for all $g \in G$. Recall that $H^1(G,A)$ can be identified with $Der(G,A)/P(G,A)$ (see for instance Chapter IV of \cite{Br}).
Thus, we have shown that for $d\geq 2$, $$Out(\GG) \cong  H^1(\Z^{d-1},{\mathcal R}_d(\Z_q)) \rtimes ( ( U({\mathcal R}_d(\Z_q))/ M) \rtimes \p).$$

If $d=2$, we have seen that $\p \cong \Z_2$, and ${\mathcal R}_2(q)=\Z_q[t, t^{-1}]$. Since the set of derivations $Der(\Z, \Z_q[t,t^{-1}]) \cong \Z_q[t,t^{-1}]$, it follows that $$Der(\Z, \Z_q[t,t^{-1}]) /P(\Z, \Z_q[t,t^{-1}]) \cong \Z_q[t, t^{-1}]/ \langle t-1  \rangle, $$ and we obtain the statement of the theorem for $d=2$.

To complete the proof of the theorem, we must show that if $d\geq 3$, then $H^1(\Z^{d-1},{\mathcal R}_d(\Z_q))=0$, or equivalently, that $ Der(\Z^{d-1}, {\mathcal R}_d(\Z_q) ) =P(\Z^{d-1}, {\mathcal R}_d(\Z_q) )$ when $d \geq 3$. To see this, let $\delta \in Der(\Z^{d-1}, {\mathcal R}_d(\Z_q) )$. Then since $\Z^{d-1}$ is abelian, for any nonzero ${\bf v} , {\bf w} \in \Z^{d-1}$, we have $\delta({\bf v})+\delta({\bf w})^{\bf v}=\delta({\bf w})+\delta({\bf v})^{\bf w}$. Thus,  $\delta({\bf v})^{\bf w}-\delta({\bf v})=\delta({\bf w})^{\bf v}-\delta({\bf w})$, which implies that there exists ${\mathcal A} \in \{ \frac{f}{g} | f, g \in \Z_q[t]\}$ such that
$$\frac{\delta({\bf v})}{ (\Pi_{i=1}^{d-1} (t+l_i)^{v_i})-1}={\mathcal A}$$ for any nonzero ${\bf v}$. If we set $(\frac{f}{g})^{\bf v}=\frac{f^{\bf v}}{g}$, then we have $\delta({\bf v})={\mathcal A}^{\bf v}-{\mathcal A} $ for any ${\bf v} \in \Z^{d-1}$. Since $d \geq 3$, we may choose $i \neq j$, $1 \leq i,j \leq d-1$, and then  since $\delta({\bf e}_i)-\delta({\bf e}_j) \in {\mathcal R}_d(\Z_q)$, we have
$$\delta({\bf e}_i)-\delta({\bf e}_j) =((t+l_i){\mathcal A}-{\mathcal A})-((t+l_j){\mathcal A}-{\mathcal A})=(l_i-l_j){\mathcal A} \in {\mathcal R}_d(\Z_q).$$ Since $l_i-l_j$ is invertible in $\Z_q$, this implies that in fact ${\mathcal A} \in {\mathcal R}_d(\Z_q))$ and hence $\delta \in P(\Z^{d-1}, {\mathcal R}_d(\Z_q) )$, as desired.
\end{proof}

In Theorem \ref{thm:S} we have completely determined $\p$,  one factor of $Out(\GG)$ when $d\geq3$. In the case that $q$ is prime, we claim that the other factor of  $Out(\GG)$, $U({\mathcal R}_d({\Z_q}))/M$, is simply $U(\Z_q)=\Z_q-\{ 0 \}$. This fact will be used in Section \ref{sec:Rinf}.

\begin{proposition}\label{prop:Rsimple}
Let $R \in  U({\mathcal R}_d({\Z_q}))$. If $q$ is prime, then $R=c \Pi _{i=1}^{d-1} (t+l_i)^{v_i}$ where $v_i \in {\mathbb Z}$ for all $i$, and $c \in {\mathbb Z}_q$ with $c \neq 0$.
\end{proposition}
\begin{proof}
Since $R$ is invertible,  $RS=1$ for some $S \in  {\mathcal R}_d({\Z}_q)$. Let $R=  \frac{f}{\Pi_{i=1}^{d-1} (t+l_i)^{m_i}}$ and let $S= \frac{g}{\Pi_{i=1}^{d-1} (t+l_i)^{n_i}}$
where $f,g \in \Z_q[t]$ and $m_i, n_i \in \Z$ with $m_i, n_i \geq 0$ for all $i$. Then $fg= \Pi_{i=1}^{d-1} (t+l_i)^{m_i + n_i}$ holds in ${\Z}_q[t]$. But if $q$ is prime then ${\Z}_q$ is a field, hence ${\Z}_q[t]$ is a unique factorization domain. Since $t+l_i$ is an irreducible polynomial, it follows that $f= b \Pi_{i=1}^{d-1} (t+l_i)^{s_i}$ for $s_i \geq 0, s_i \in {\mathbb Z}$ and $b \neq 0$. Thus,
$R= c\Pi_{i=1}^{d-1} (t+l_i)^{v_i}$, where $v_i=s_i-m_i \in {\mathbb Z}$ and $c \in {\mathbb Z}_q$, $c \neq 0$.
\end{proof}

\section{Counting twisted conjugacy classes in $\GG$}\label{sec:Rinf}

There are a variety of techniques in the literature for counting the number of twisted conjugacy classes of a group homomorphism; some apply to endomorphisms or homomorphisms, whereas we are concerned only with automorphisms.   Let $R(\vp)$ denote the cardinality of the set $\mathcal R(\vp)$ of $\vp$-twisted conjugacy classes. We say that a group $B$ has property $\Rinf$ if any $\vp \in Aut(B)$ has $R(\vp) = \infty$. When a group $B$ can be expressed via a short exact sequence in which the kernel is a characteristic subgroup, any automorphism of $B$ yields a commutative diagram:
\begin{equation}\label{diagram:countingconjugacy}
\begin{CD}
    1   @>>> A    @>{i}>>  B @>{p}>>      C @>>> 1 \\
    @.  @V{\ph'}VV      @V{\ph}VV   @V{\overline \ph}VV @.\\
    1   @>>> A    @>{i}>>  B @>{p}>>      C @>>> 1
 \end{CD}
\end{equation}
Here $\vp \in Aut(B)$, and $\vp'$ and $\Op$ are the induced automorphisms on the kernel and quotient respectively.  A much used technique for counting the number of twisted conjugacy classes of a homomorphism $\vp$ is to relate $R(\vp)$ to $R(\Op)$ and $R(\vp')$. In the case where $\vp$ is an automorphism and $C$ is characteristic, the relationship is quite simple. The following result is straightforward; a proof of part (a) and additional background are given in  \cite{Wo}. We include a proof for completeness.

\begin{lemma}\label{lemma:reid}
Given the commutative diagram labeled \eqref{diagram:countingconjugacy} above,
\begin{enumerate}
\item if $R(\overline \ph)=\infty$ then $R(\ph)=\infty$,

\item if $R(\vp') = \infty$ and $Fix(\Op)=1$, then $R(\ph)=\infty$.
\end{enumerate}
\end{lemma}
\begin{proof}
The two statements follow directly from the following two facts:
\begin{enumerate}
\item If $b$ and $b'$ are $\vp$-twisted conjugate in $B$, then $p(b)$ and $p(b')$ are $\Op$-twisted conjugate in $C$.
\item If $i(a)$ and $i(a')$ are $\vp$-twisted conjugate in $B$, then $a$ and $a'$ are $\vp'$-twisted conjugate in $A$.
\end{enumerate}
The first fact is easily verified, for if $x b\vp(x^{-1})= b' $ for some $x \in B$, then applying the projection we have $p(x)p(b)\Op( p(x)^{-1})= p(b')$. Then the fact that $p$ is surjective shows that if $R(\overline \ph)=\infty$ then $R(\ph)=\infty$.
For the second fact, if $x i(a)\vp(x)^{-1}= i(a')$ for some $x \in B$, projecting via $p$ shows that $\Op(p(x))=p(x)$. But since $Fix(\Op)=1$, then $p(x)=1$, so $x=i({\overline a})$ for some $ {\overline a}\in A$. Thus,  ${\overline a}a \vp'({\overline a}^{-1})= a'$. Therefore  if $R(\vp') = \infty$ and $Fix(\Op)=1$, then $R(\ph)=\infty$.
\end{proof}
The following lemma will also be used repeatedly.
\begin{lemma}\label{lemma:determinant}
Let $\beta \in Aut(\Z^r)$ for $r \in \Z^+$.  Then $R(\beta) < \infty$ iff $Det(Id - \beta) \neq 0$ iff $Fix(\beta)$ is trivial.
\end{lemma}

\begin{proof}
Since $R(\beta)$ is the number of orbits of the action $\sigma \cdot \alpha \mapsto \sigma \alpha \vp (\sigma)^{-1}$ and $\Z^r$ is abelian, it follows that $R(\beta)$ is the index of the subgroup $(Id -\beta)\Z^r$ in $\Z^r$. Thus, $R(\beta)<\infty$ iff $(Id - \beta)$ has full rank in $\Z^r$. Since we can represent $Id-\beta$ by an $r\times r$ integral matrix (with respect to some basis), it follows that $R(\beta)<\infty$ if and only if $Det(Id -\beta)\neq 0$. Note that $Det(Id -\beta)\neq 0$ if and only if $Ker(Id -\beta)=0$, which means that $(Id -\beta){\bf x}={\bf 0}$ has only trivial solutions in $\Z^r$, that is, the subgroup $Fix(\beta)=\{{\bf x}\in \Z^r \mid \beta({\bf x})={\bf x}\}$ is trivial.
\end{proof}

We now show that for $d \geq 3$, the group $\GG$  has property $\Rinf$.  Note that when $d=2$, the lamplighter group $L_q=\Gamma_2(q)$ has property $\Rinf$ if and only if $(q,6) \neq 1$.

\begin{theorem}\label{thm:Rinf} The group $\GG$ has property $\Rinf$ for all $d \geq 3$.
\end{theorem}

\begin{proof} Let $\vp \in Aut(\GG)$ be any automorphism, and recall that $\vp$ has the form $(\delta, (R, \beta)) \in Der(\Z^{d-1}, {\mathcal R}(\Z_q)) \rtimes ( U({\mathcal R}(\Z_q)) \rtimes \p)$, and in the notation of Lemma \ref{lemma:reid}, we have $\Op=\beta$ , and $\vp'=\alpha_{R, \beta}$.

In almost all cases, we  show that $R(\Op)=\infty$, which then implies that $R(\vp)=\infty$  by the first  conclusion of Lemma \ref{lemma:reid}. In the few remaining cases, we will show that $R(\vp')=\infty$ and $Fix(\Op)=1$, and apply the second conclusion of  Lemma \ref{lemma:reid} to obtain $ R(\vp)=\infty$. Throughout this discussion, we identify $\beta=\Op$ with its matrix representation.

Now by  Lemma \ref{lemma:determinant}, $R(\beta)=\infty$ if and only if $Det(Id-\beta) =0$.
When $d>3$  it follows from Theorem \ref{thm:S} that $\beta$ is either the identity or a permutation matrix corresponding to a cycle of length $d-1$.  If  $d=3$, there are four additional possibilities for $\beta$ other than the identity or the transposition.
 If $\beta$ corresponds to a cycle of length $d-1$, each column  of $Id-\beta$ has only two nonzero entries, $1$ and $ -1$. Adding all rows to the first yields a matrix with a first row which has all zeroes, hence, $det(Id-\beta)=0$.

If $d=3$ and $\beta$ is one of the two matrices
  $\left( \begin{array}{cc} 1 & 0 \\ -1 & -1 \end{array} \right)$ or  $\left( \begin{array}{cc} -1 & -1 \\ 0 & 1 \end{array} \right)$, one easily verifies that $det(Id-\beta)=0$. Thus, $R(\beta)=\infty$ for every $\beta$ with the exception of
 $\beta =\left( \begin{array}{cc} -1 & -1 \\ 1 & 0 \end{array} \right)$ or $\beta = \left( \begin{array}{cc} 0 & 1 \\ -1 & -1 \end{array} \right)$, which occur for some $\vp \in Aut(\Gamma_3(2))$. So for every $\vp$ except these two special cases, we have $R(\vp)=\infty$ by the first conclusion of Lemma \ref{lemma:reid}.
In the two special cases,  $Det( Id-\beta) \neq 0$, so
$Fix(\beta)=\{ {\bf 0} \}$. We claim that in both cases $R(\alpha_{R, \beta})=\infty$, and so by part (2) of Lemma \ref{lemma:reid} we have $R(\vp)=\infty$ in these cases as well.

Let $\beta$ be one of the two matrices above. We claim that $\alpha_{R, \beta}^3=Id$. Since $\alpha_{R, \beta}
(S)=RS^{\beta}$ for any $S \in {\mathcal R}_3(\Z_2)$, $$\alpha_{R, \beta}^3(S)=(R R^{\beta} R^{\beta^2}) S^{\beta^3}.$$ We know from Proposition \ref{prop:Rsimple} that since $2=d-1$ is prime, $R=t^k (1+t)^l$ for some $k, l \in \Z$. Note that such an expression is unique, as it is easy to verify that $t^a (1+t)^b=1$ implies that $a=b=0$.  Let ${\bf w}=(k,l)$.  Then \begin{align*}R R^{\beta} R^{\beta^2}&=1^{\bf w}( 1^{\bf w})^{\beta} (1^{\bf w})^{\beta^2}=1^{\bf w}( 1^{\beta})^{\beta({\bf w})}  (1^{\beta^2}) ^{\beta^2({\bf w})}\\&=1^{\bf w}1^{\beta({\bf w})}1^{\beta^2({\bf w})}=1^{(id+\beta+\beta^2)({\bf w})}.\end{align*} But for either choice of $\beta$,  $Id+\beta+\beta^2=0$, so $R R^{\beta} R^{\beta^2}=1$. Moreover,
$\beta^3=Id$, so $\alpha_{R, \beta}^3(S)=S$ as desired.

Two elements $A, B \in \GG'$ are $\alpha_{R, \beta}$ twisted conjugate if $A-B=P - \alpha_{R, \beta}(P)$ for some $P \in \GG'$.
Note that  $P - \alpha_{R, \beta}(P)=P +\alpha_{R, \beta}(P)$  and $A-B=A+B$ since the ring coefficients are in $\Z_2$.
If $A \neq B$, we claim that if $A$ and $B$ are both fixed by $\alpha_{R, \beta}$, then  $A$ and $B$ are  not $\alpha_{R, \beta}$ twisted conjugate. To see this, suppose to the contrary they are, so $A+B=P + \alpha_{R, \beta}(P)$ for some $P \in \GG'$. Then since  $\alpha_{R, \beta}$ fixes $A+B$, it fixes $P + \alpha_{R, \beta}(P)$ as well. Hence, $\alpha_{R, \beta}^2(P)=P$, but since $\alpha_{R, \beta}^3=id$, this implies that $P = \alpha_{R, \beta}(P)$, which implies that $A=B$, a contradiction. Thus, to prove that $R(\alpha_{R, \beta})=\infty$, it suffices to produce an infinite sequence $S_1, S_2, \ldots$ of distinct elements of $\GG'$, each of which is fixed by $\alpha_{R, \beta}$.

We construct such a sequence for $\beta= \left( \begin{array}{cc} -1 & -1 \\ 1 & 0 \end{array} \right)$. Choose $m,n>0$ satisfying $n>k+l$ and  $m>k$, and let ${\bf v}=(m,n)$. We  define $$S_j= 1^{j{\bf v}}+ \alpha_{R, \beta}( 1^{j{\bf v}})+ \alpha_{R, \beta}^2( 1^{j{\bf v}})~{\text for}~ j \in \Z^+.$$ It is clear that each $S_j$ is fixed by $ \alpha_{R, \beta}$, so it only remains to establish that $S_i \neq S_j$ if $i \neq j$. We prove this by computing ${\mathcal LS}_3(S_j)$, from which we deduce that $P_3(S_j)$ has lowest degree term of degree $-j(n+m)$ in the variable $t^{-1}$. Thus, $P_3(S_i) \neq P_3(S_j)$ if $i \neq j$, and hence $S_i \neq S_j$.

A straightforward computation yields
\begin{align*} 1^{j{\bf v}} &= t^{jm} (1+t)^{jn}~, \\   \alpha_{R, \beta}( 1^{j{\bf v}}) &=t^{k-j(m+n)}(1+t)^{ l+jm} , \\ \alpha_{R, \beta}^2( 1^{j{\bf v}}) &=t^{jn-l}(1+t)^{k+l-j(m+n)}.\end{align*}

Therefore, ${\mathcal LS}_3 (1^{j{\bf v}})$ has lowest degree term (in the variable $t^{-1}$) of degree $-j(m+n)$, and since this degree is negative this term is present in $P_3(S_j)$. Now  ${\mathcal LS}_3 (\alpha_{R, \beta}( 1^{j{\bf v}}))$ has lowest degree term of degree $(jn)-(k+l)$, which is strictly positive according to the choice of $m$ and $n$. Thus, ${\mathcal LS}_3 (\alpha_{R, \beta}( 1^{j{\bf v}}))$ contributes no terms to $P_3(S_j)$. Finally,  ${\mathcal LS}_3 (\alpha_{R, \beta}^2( 1^{j{\bf v}}))$ has lowest degree term with degree $jm-k>0$, so ${\mathcal LS}_3 (\alpha_{R, \beta}^2( 1^{j{\bf v}}))$ contributes no terms to $P_3(S_j)$. Hence, $P_3(S_j)$ has lowest degree term of degree $-j(n+m)$, as claimed.
A similar argument shows that if $\beta = \left( \begin{array}{cc} 0 & 1 \\ -1 & -1 \end{array} \right)$, then $R(\alpha_{R,\beta})=\infty$ as well.  Thus, for all $d \geq 3$, the group $\GG$ has property $\Rinf$.
\end{proof}

\section{Remarks on Baumslag's Metabelian Group}

We can mimic the arguments above to compute the automorphism group of Baumslag's metabelian group, which has presentation
$$BMG= \langle a,s,t | st=ts, [a,a^t],aa^s=a^t \rangle= \Z[t,(t+l_1)^{-1},(t+l_2)^{-1}] \rtimes \Z^2 ,$$ where $l_1=0$ and $l_2=\pm1$. Note that these are the only choices for a pair $l_1=0$ and $l_2$ where $l_2-l_1$ is invertible in $\Z$.  Letting ${\mathcal R}_d(\Z)=\Z[t,(t+l_1)^{-1},(t+l_2)^{-1}]$, we obtain the following theorem, whose proof is analogous to that of Theorem \ref{thm:AutGamma}.

\begin{theorem}\label{thm:bmg}
$$Aut(BMG)=Der( \Z^2, {\mathcal R}_d(\Z)) \rtimes (U({\mathcal R}_d(\Z)) \rtimes \p),$$ where $$\p = \left\{ \left( \begin{array}{cc} 1 & 0 \\ 0& 1 \end{array} \right), \left( \begin{array}{cc} -1 & -1 \\ 0 & 1 \end{array} \right) \right\}~{\text  if}~ l_2=1,$$ and $$\p = \left\{ \left( \begin{array}{cc} 1 & 0 \\ 0& 1 \end{array} \right), \left( \begin{array}{cc} 1 & 0\\ -1 & -1 \end{array} \right) \right\}~{\text if}~ l_2=-1.$$
\end{theorem}

Similarly, we obtain an analogue of Theorem \ref{thm:out}.

\begin{theorem}\label{out:bmg}
$$Out(BMG)=  (U({\mathcal R}_d(\Z))/ M)\rtimes \p,$$
where $M=\{ t^{x_1}(t+1)^{x_2} \}$ is the set of monomials with coefficient one, and $\p \cong \Z_2$ is as above.
\end{theorem}

Finally, following the reasoning in Section \ref{sec:Rinf}, replacing Theorem \ref{thm:AutGamma} with Theorem \ref{thm:bmg}, yields the following theorem.

\begin{theorem}\label{thm:bmg:rinf}
Baumslag's metabelian group has property $R_{\infty}$.
\end{theorem}

\bibliographystyle{plain}
\bibliography{refs}

\end{document}